\documentclass[12pt]{amsart}

\usepackage{amsfonts,amsmath,amssymb,amsthm}
\usepackage{hyperref}
\hypersetup{breaklinks=true}
\usepackage{rotating}
\usepackage{array}

\numberwithin{equation}{section}
\theoremstyle{plain}
\newtheorem{theorem}[equation]{Theorem} 
\newtheorem{lemma}[equation]{Lemma}    
\newtheorem{proposition}[equation]{Proposition}   
\newtheorem{corollary}[equation]{Corollary}

\newcolumntype{M}[1]{>{\raggedright}m{#1}}

\newcounter{thm}

\newtheorem{main_theorem}[thm]{Theorem}
\newtheorem{main_corollary}[thm]{Corollary}

\theoremstyle{remark}
\newtheorem*{remark*}{Remarks} 
\newtheorem*{example*}{Example}

\theoremstyle{definition}
\newtheorem{definition}[equation]{Definition}
\newtheorem{remark}[equation]{Remark}

\newcommand{\Ker}{\operatorname{Ker}}

\newcommand{\restr}{\mbox{\Large \(|\)\normalsize}}

\begin{document}

\title[$L^p$-cohomology: critical exponents and rigidity]
{$L^p$-cohomology\\ for higher rank spaces and groups:\\ critical exponents and rigidity}
 
\author{Marc Bourdon and Bertrand R\'emy} 
\maketitle

\begin{abstract} 
We initiate the investigation of critical exponents (in degree equal to the rank) for the vanishing of $L^p$-cohomology of higher rank Lie groups and related manifolds. 
We deal with the rank 2 case and exhibit such phenomena for ${\rm SL}_3({\bf R})$ and for a family of 5-dimensional solvable Lie groups. We use the critical exponents to compare the groups up to quasi-isometry. 
This leads us to exhibit a continuum of quasi-isometry classes of rank 2 irreducible solvable Lie groups of non-positive curvature. Along the proof, we provide a detailed description of the $L^p$-cohomology of the real and complex hyperbolic spaces. It is then combined with a spectral sequence argument, to derive our higher-rank results. 

\vspace*{2mm} \noindent{2010 Mathematics Subject Classification: } 20J05, 20J06, 22E15, 22E41, 53C35, 55B35, 57T10, 57T15.
%
%
%
%

\vspace*{2mm}

\noindent{Keywords and phrases: } $L^p$-cohomology, Lie~group, symmetric space, quasi-isometric invariance, spectral sequence, cohomology (non-)vanishing. 
\end{abstract}

\setlength{\parskip}{\smallskipamount}

\tableofcontents

\setlength{\parskip}{\medskipamount}


\section*{Introduction}
\label{1}

\subsection*{Overview}\label{overview}
$L^p$-cohomology, with $p \in (1, +\infty)$, provides a family of large scale geometry invariants for metric spaces and groups. 
It has many incarnations (such as asymptotic $L^p$-cohomology, or group $L^p$-cohomology via continuous cohomology of locally compact groups) which are all comparable to one another under suitable, not so demanding, conditions. 
Each variant brings its own insights: for instance asymptotic $L^p$-cohomology shows that $L^p$-cohomology is an invariant under quasi-isometry (in fact, under coarse isometry), and continuous cohomology allows one to use standard algebraic tools such as spectral sequences. 
In the present paper, we are interested in the de Rham $L^p$-cohomology, which we denote by $L^p \mathrm{H}^*_{\mathrm{dR}}$. Roughly speaking, we are dealing with forms satisfying, together with their differentials, $L^p$-integrability conditions with respect to measures given by suitable Riemannian metrics. 

References for $L^p$-cohomology include \cite{G} for a general overview, \cite{Pa95, Genton, Sequeira} for its invariance under quasi-isometry, \cite{Elek, CT, SaSc, BR} for group $L^p$-cohomology, \cite{BR, BR2, LN} for spectral sequences and applications, \cite{Pa99, P2, P1, Pa09, Sequeira} for de Rham cohomology of Lie groups. We also notice that \cite[§6]{BR2} contains a synthetic presentation of (some) of the several aspects of $L^p$-cohomology, as well as a description and comparison of their properties. 

$L^p$-cohomology of a (connected) Lie group is better understood in rank 1 situations (\footnote{The {\it rank} of a Lie group is the dimension of its asymptotic cones; it is therefore a quasi-isometric invariant. For semisimple Lie groups, this notion coincides with the $\mathbf R$-rank. When the group is simply connected and solvable, its rank is equal to the codimension of its exponential radical \cite[Corollary 1.3]{CornulierTop}. When the group admits a left-invariant Riemannian metric of non-positive curvature, its rank is the same as  the maximal dimension of a totally geodesic Euclidean subspace \cite{AW}. We thank Yves Cornulier for providing us with this definition.}), when contractions and negative curvature arguments can be used to perform some computations. 
In particular, critical exponent phenomena with respect to $p$  for vanishing vs non-vanishing of $L^p$-cohomology in degree 1, can be sometimes exhibited.
An iconic example of this phenomenon is provided by the following family of solvable Lie groups: for $\lambda \geqslant 1$, let $H_\lambda$ be the semi-direct product $H_\lambda = \mathbf R \ltimes_\lambda \mathbf R^2$, where $\mathbf R$ acts on $\mathbf R^2$ via the $1$-parameter group of automorphisms
$t \mapsto e^{-tA _\lambda}$, with 
$A_\lambda = \begin{pmatrix}1 & 0 \\ 0 & \lambda \end{pmatrix}$.
When $\lambda =1$, the group $H_\lambda$ is naturally isometric to the real hyperbolic $3$-space $\mathbb H _{\mathbf R} ^3$. In general $H _\lambda$ belongs to the family of the so-called {\it Heintze groups}, {\it i.e.\!} of the Lie groups that admit a left-invariant negatively curved Riemannian metric \cite{Hze}.
Pansu’s Theorem \cite{P1} shows that $L^p \mathrm{H}^1 _{\mathrm{dR}}(H _\lambda)$ vanishes for $p \in (1; 1+\lambda)$, and does not vanish for $p > 1+\lambda$. In other words, $1+\lambda$ is a critical exponent of the first $L^p$-cohomology of $H_\lambda$. As a consequence, the groups $H_\lambda$ are pairwise non-quasiisometric.
More generally, every Heintze group admits an explicit critical exponent in degree $1$ \cite{P2, CT}.


$L^p$-cohomology of higher rank Lie groups has attracted less attention so far. As a first step, it would be desirable to have a better understanding of $L ^p$-cohomology in degree equal to the rank.  Indeed, the first (reduced) $L^p$-cohomology of higher rank Lie groups is known to vanish for every $p$ (\footnote{ More precisely, apart from those which are quasi-isometric to an Heintze group, the reduced first $L^p$-cohomology of every Lie group vanihes for every $p$, and the non-reduced one vanishes if, and only, if the group is non-amenable or non-unimodular \cite{P2, CT}}); and vanishing for every $p$ is expected to remain true in any degree below the rank – at least for semisimple Lie groups  \cite[p. \!253]{G}  \cite{LN}. In degree equal to the rank, critical exponents are known to exist for several higher rank real Lie groups, including all the semisimple ones \cite{BR2}; but their values has not been determined yet.

In the present paper, we study the second $L^p$-cohomology of solvable Lie groups of rank $2$. 
More precisely, we exhibit, for some groups of this type, a critical exponent in degree $2$. We then use these critical exponents to derive a quasi-isometric rigidity result.
 
\subsection*{A family of solvable Lie groups}

We consider the solvable Lie groups of the form $S_\alpha = \mathbf R^2 \ltimes _\alpha \mathbf R^3$, where $$\alpha : \mathbf R^2 \to \{\mathrm{diagonal~~automorphisms~~of~~}\mathbf R^3\}$$ is a Lie group morphism. We denote by $\varpi _i \in (\mathbf R^2)^*$ ($i = 1, 2, 3$) the {\it weights} associated to $\alpha$, {\it i.e.\!} the linear forms such that
$\alpha = e ^{\mathrm{diag}(\varpi _{1}, \varpi _2 , \varpi _{3})}.$

We shall let $\mathcal S ^{2, 3}_{\mathrm{straight}}$ denote the set of the groups $S_\alpha$ whose weights enjoy the following two properties: 
\begin{itemize}
 \item they generate $(\mathbf R^2)^*$, 
 \item they belong to an affine line (necessarily disjoint from $0$).
\end{itemize}

Every group $S _\alpha \in \mathcal S^{2,3} _{\mathrm{straight}}$ is of rank $2$, admits a left-invariant Riemannian metric of non-positive curvature, and is irreducible provided its weights are pairwise distinct -- see Proposition \ref{appendix_prop} in the appendix for a proof of these properties (in a wider generality). Therefore, in some sense, $\mathcal S^{2,3} _{\mathrm{straight}}$ appears as the simplest family of irreducible higher rank non-positively curved Lie groups.

Our first result exhibits a critical exponent of the second $L^p$-cohomology of the groups that belong to $\mathcal S^{2,3}_{\mathrm{straight}}$. It answers partially a question of Cornulier.

\begin{main_theorem}
\label{S_mu-thm}
Let $S_\alpha \in \mathcal S ^{2,3}_{\mathrm{straight}}$. Its weights $\varpi _{1}, \varpi _2, \varpi _3$ belong to a line. Since permuting the coordinates of $\mathbf R^3$ preserves the isomorphism class of $S_\alpha$, we can (and will) assume that the algebraic distances between the weights on the (suitably oriented) line satisfy: $0 \leqslant \varpi _2-\varpi _3 \leqslant \varpi _1-\varpi _2$. Set
$$p_\alpha := 1+\frac{\varpi _1-\varpi _3}{\varpi _1-\varpi _2} \in [2;3].$$
Then 
$L^p {\rm H}^2_{\rm dR}(S_\alpha) = \{0\}$ for $p \in (1; p_\alpha) \setminus \{\frac{3}{2}\}$, and $L^p {\rm H}^2_{\rm dR}(S_\alpha) \neq \{0\}$ for $p \in (p_\alpha; +\infty) \setminus \{3\}$.
\end{main_theorem}

Observe that two triples $(\varpi _{1}, \varpi _{2}, \varpi _{3})$ and $(\varpi '_{1}, \varpi '_{2}, \varpi '_{3})$ -- of distinct elements in $(\mathbf R ^{2})^*$ that are aligned on lines disjoint from $0$ -- belong to the same $\mathrm{GL} _{2} (\mathbf R)$-orbit if, and only, if
$$\frac{\varpi _1-\varpi _3}{\varpi _1-\varpi _{2}} = \frac{\varpi '_1-\varpi '_3}{\varpi '_1-\varpi ' _{2}}.$$
Since precomposing $\alpha$ by an element of $\mathrm{GL} _{2} (\mathbf R)$ preserves the isomorphism class of $S _\alpha$, and since de Rham $L^p$ -cohomology is a quasi-isometry
invariant among Lie groups that are diffeomorphic to $\mathbf R^D$ \cite{Pa95},\cite[Appendice]{BR2}, Theorem \ref{S_mu-thm} admits the following rigidity consequence:

\begin{main_corollary}\label{intro_cor1}
Among the groups in $\mathcal S^{2,3} _{\mathrm{straight}}$, any two of them are quasi-isometric if, and only if, they are isomorphic.
\end{main_corollary}

Theorem \ref{S_mu-thm} also yields:

\begin{main_corollary}
There exists a continuum of quasi-isometry classes of rank $2$ solvable irreducible non-positively curved Lie groups.
\end{main_corollary}

\subsection*{The symmetric space ${\rm SL}_3({\mathbf R})/{\rm SO}_3({\mathbf R})$}
Going back to semisimple groups, the same approach -- which we describe below in more details -- enables us to exhibit another critical exponent in degree 2:  

\begin{main_theorem}
\label{SL3-theorem}
Let $S$ be the symmetric space ${\rm SL}_3({\mathbf R})/{\rm SO}_3({\mathbf R})$, or equivalently the Borel subgroup of ${\rm SL}_3({\mathbf R})$ consisting of upper triangular matrices with positive diagonal. Then $L^p \mathrm{H}_{\mathrm{dR}}^2 (S)$ is zero for $p \in (1; 2) \setminus \{\frac{4}{3}\}$, and non-zero for $p \in (2; +\infty) \setminus \{4\}$.
\end{main_theorem}

By Iwasawa decomposition, the symmetric space $S$ is naturally isometric to the Borel subgroup of ${\rm SL}_3({\mathbf R})$; the latter is a specific rank 2 solvable group, namely the semidirect product of ${\bf R}^2$ and of the Heisenberg group in dimension 3.

\subsection*{About the proofs} In both instances of the solvable groups dealt with in the above theorems, we can decompose the action of ${\bf R}^2$ on the 3-dimensional subgroup ${\bf R}^3$ (resp. \!on the Heisenberg group in dimension 3, which we denote by $\mathrm{Heis}(3)$) into two steps. 
In a first step, one factor $\bf R$ of $\bf R^2$ acts on ${\bf R}^3$ (resp. \!on $\mathrm{Heis}(3)$) so that the intermediate (rank 1) semidirect product is a non-unimodular solvable group isometric to the real (resp. \!complex) hyperbolic space of real dimension 4. 
Then, as a second step, we consider the action of the second factor $\bf R$ of ${\bf R}^2$ and use a spectral sequence argument, together with the fact that we understand in detail the cohomology of the intermediate 4-dimensional group of the first step. 
Thus, at this stage, proving the non-vanishing of the considered $L^p$-cohomology amounts to showing that some de Rham classes on the rank 1 group satisfy a certain $L^p$-integrability condition (see Section \ref{ss - non-vanishing}, and Relation \ref{SL3-iso} in Section \ref{SL3-proof}). 
The vanishing part requires to use a Poincar\'e duality argument in order to show the requested non-integrability of the relevant de Rham classes. 

The main result about the rank 1 intermediate solvable groups above is Theorem \ref{Lie-thm}. 
It provides a partial description of the $L^p$-cohomology of Lie groups containing a suitable 1-parameter subgroup of (semi) contractions acting on its complement. The obtained description complements some previous results of Pansu \cite[Sections 9 and 10]{Pa99}; 
we call it a {\it strip decomposition} since its hypotheses are stated as (double) inequalities that must be satisfied by the exponent $p$ with respect to quantities depending on the degree $k$ of the cohomology and on the infinitesimal eigenvalues of the contraction group. 
The conclusions deal with the following properties: vanishing, Hausdorff property, density of some explicit subspaces of closed forms, and finally Poincar\'e duality realized at infinity, {\it i.e.} on the group-theoretic complement of the contraction group (which is a Lie group seen as a boundary of the ambient group). 

The statement of Theorem \ref{Lie-thm} is applied to the special semidirect products $\mathbf R \ltimes \mathbf R^n$ and $\mathbf R \ltimes \mathrm{Heis}(2m-1)$, that are isometric to the real hyperbolic space $\mathbb H ^{n+1}_{\mathbf R}$ and to the complex one $\mathbb H ^{m}_{\mathbf C}$. We obtain in this way a rather precise description of the $L^p$-cohomology of $\mathbb H ^{n+1}_{\mathbf R}$. In the case of the complex hyperbolic space, the information given by Theorem \ref{Lie-thm} is fragmented, and a substantial additional amount of work dealing with Heisenberg groups of arbitrary dimension, elaborating on ideas due to Rumin \cite{Rumin} 
and Pansu \cite{Pa09}, is required.

Again, from a technical point of view, the paper deals with de Rham cohomology only, and some of our results are valid for Riemannian manifolds endowed with a suitable contracting vector field, even though the main applications are relevant to the Lie group situation. This applies in particular to the main new technical result (Theorem \ref{boundary-thm}) which translates the Poincar\'e duality in terms of currents on the ``boundary''.

Let us finish this introduction with some remarks. 

\begin{remark}\label{remark_1} Pansu has already used $L^p$-cohomology to show that the groups $H_\alpha := \mathbf R \ltimes _\alpha \mathbf R ^n$, with $$\alpha (t)= \mathrm{diag}(e^{-\alpha _1 t}, \dots, e^{-\alpha _n t})~~\ \mathrm{and}~~\  
1= \alpha _1 \leqslant \alpha _2 \leqslant \dots \leqslant \alpha _n,$$ form a continuous family of pairwise non-quasiisometric Heintze groups \cite[Corollary 2]{Pa99}, \cite{Sequeira}. This result has been generalized by Xie \cite[Corollary 1.3]{Xie} to non-diagonal automorphisms, by using more geometric methods.
\end{remark}

\begin{remark} It is a natural to ask whether the rigidity statement of Corollary \ref{intro_cor1} remains true for other family of groups, like the families $\mathcal S^{r ,n}$ considered in the appendix. This question has been already studied in some special cases, {\it e.g.\!} for the family evoked just above in Remark \ref{remark_1}. In \cite{EskinFisherWhyte}, the authors study the groups $\mathrm{Sol} _\lambda := \mathbf R \ltimes_\lambda \mathbf R^2$, where $\lambda \geqslant 1$, and where $\mathbf R$ acts on $\mathbf R ^2$ via $t \mapsto e ^{tB _\lambda}$, with $B _\lambda = \begin{pmatrix}1 & 0 \\ 0 & -\lambda \end{pmatrix}$. They show that these groups are pairwise non-quasiisometric. Observe that the groups $\mathrm{Sol} _\lambda$ are rank $1$ solvable Lie groups that do not carry any non-positively curved left-invariant Riemannian metric -- see Proposition \ref{appendix_prop}(3). 

In \cite[Corollary 5.3.7]{Pen2}, Peng establishes that if $G= \mathbf R ^r \ltimes _\varphi \mathbf R ^n$ and $G' = \mathbf R ^{r'} \ltimes _ {\varphi '} \mathbf R ^{n'}$ are two quasi-isometric, \emph{unimodular}, non-degenerate split abelian-by-abelian solvable Lie groups, then there exists an isomorphim $f: \mathbf R ^r \to \mathbf R ^{r'}$ such that $\varphi$ and $\varphi' \circ f$ have the same Jordan form. Again, due to the unimodularity, these groups do not carry any non-positively curved left-invariant Riemannian metric. 
\end{remark}

\begin{remark} Examples of non-positively curved, rank $2$, {\it reducible}, solvable Lie groups, include the groups $\mathbf R \times H _\lambda$ (in dimension $4$), and the groups $\mathbb H ^2 _{\mathbf R} \times H _\lambda$ (in dimension $5$); where $H_\lambda$ ($\lambda \geqslant 1$) is the Heintze group defined in the overview. As a consequence of a general quasi-isometric rigidity theorem for product metric spaces \cite[Theorem B]{KKL}, all these groups are pairwise non-quasiisometric.
\end{remark}

\begin{remark} Our results on de Rham $L^p$-cohomology of hyperbolic spaces can be compared with Borel's on $L^2$-cohomology of symmetric spaces \cite{Borel}. 
It turns out that for complex hyperbolic spaces our results are complementary in the sense that the exponent $p=2$ is never contained in the interior of the strips we distinguish.
Nevertheless, for $\mathbb H^{m} _{\bf C}$ it is in the closure (and in the middle) of the union $]2{m \over m+1} ; 2[ \,\, \sqcup \,\, ]2 ; 2{m \over m-1}[$ of two critical strips. 
For $p$ in the interior of each segment, our Theorem \ref{complex-thm} says that $L^p{\rm H}^m_{\rm dR}(\mathbb H^{m} _{\bf C})$ is Hausdorff and non-zero, while Theorem A of [loc.\! cit.] says that $L^2{\rm H}^m_{\rm dR}(\mathbb H^{m} _{\bf C})$ is Hausdorff and non-zero. Moreover it describes the latter space in representation-theoretic terms.
For real hyperbolic spaces $\mathbb H^{n+1} _{\bf R}$, we have to distinguish two cases according to the parity of $n$. 
When $n$ is odd, our Theorem \ref{real-thm} recovers Theorem A(i) of [loc.\! cit.], saying that $L^2{\rm H}^\bullet_{\rm dR}(\mathbb H^{n+1} _{\bf R})$ is Hausdorff and concentrated in degree ${n+1 \over 2}$.
When $n$ is even, Theorem B of [loc.\! cit.] complements our result, saying that $L^2\overline{{\rm H}^\bullet_{\rm dR}}(\mathbb H^{n+1} _{\bf R})$ is zero and $L^2{\rm H}^\bullet_{\rm dR}(\mathbb H^{n+1} _{\bf R})$ is not Hausdorff in degree ${n \over 2} + 1$. 
\end{remark}

\subsection*{Structure of the paper} 
Section \ref{s - currents} introduces currents in the context of $L^p$-cohomology; it also recalls Poincar\'e duality for the reduced variant of it. 
Section \ref{s - flows} introduces flows with suitable contraction properties on manifolds; it describes their effects on $L^p$-cohomology and introduces a version of Poincar\'e duality involving currents on the "boundary" of such a manifold. 
In Section \ref{s - Lie}, the situation is specialized to the case of Lie groups; the existence of a suitable 1-dimensional (semi) contracting group leads to the strip description of the $L^p$-cohomology of the groups under consideration. 
In Section \ref{s - real hyp sp}, we apply the result of the previous section to deduce the description of the $L^p$-cohomology of real hyperbolic spaces. 
Section \ref{s - non QI solvable} focusses on the proof of Theorem \ref{S_mu-thm} about the second $L^p$-cohomology of the groups $S _\alpha$; this is where we determine our first critical exponent.
Section \ref{s - complex hyp sp} provides a description of the $L^p$-cohomology of complex hyperbolic spaces; this requires more care than for the real case, and in particular it leads to an intensive use of Heisenberg groups. 
In Section \ref{s - sl3}, using the same strategy as in Section \ref{s - non QI solvable}, we determine our second higher-rank critical exponent, this time for the symmetric space 
$\mathrm{SL}_3 (\mathbf R)/{\rm SO}_3({\mathbf R})$. At last, an appendix deals with the groups $S _\alpha$ in a wider generality, and establishes some of their basic properties.

\subsection*{Acknowledgements} We would like to thank Gabriel Pallier for several helpful discussions and references on the subject of this paper. Special thanks to Yves Cornulier who drawn our attention to the groups $S _\alpha$, and asked several questions that have motivated this work. We also thank him for his numerous comments on a first version of the paper.
M.B.\! was partially supported by the Labex Cempi.

\section{Currents and $L^p$-cohomology}
\label{s - currents} 

In this section, we give a quick presentation of de Rham $L^p$-cohomology and related topics.

\subsection{Currents}
\label{current-sec} 
Currents play a central role in classical de Rham cohomology. 
We recall some of the basic definitions and properties useful for the $L^p$ variant (see \cite{DS} for more informations).

Let $M$ be a $C^\infty$ orientable $D$-manifold without boundary.
For $k \in \mathbf Z$, let $\Omega^k(M)$ be the space of $C^\infty$ differential $k$-forms on $M$, and let $\Omega _c^k (M)$ be the space of compactly supported $C^\infty$ differential $k$-forms, endowed with the $C^\infty$ topology. 
As usual we set $\Omega^k(M)=\Omega _c^k (M) = \{0\}$ for $k<0$. 

A \emph{$k$-current} on $M$ is by definition a continuous real valued linear form on $\Omega _c^{D-k} (M)$. 
We denote by $\mathcal D '^k(M)$ the space of $k$-currents on $M$ endowed with the weak*-topology. 

To every $\omega \in \Omega^k(M)$, one associates the $k$-current $T_\omega$ defined by $T_\omega (\alpha):= \int _M \omega \wedge \alpha$. 
This defines an embedding of $\Omega^k(M)$ into $\mathcal D '^k(M)$, whose image is known to be dense.
The \textit{differential} of a $k$-current $T$ is the $(k+1)$-current $dT$ defined by $dT(\alpha):= (-1)^{k+1} T(d\alpha)$, for every $\alpha \in \Omega _c^{D-k-1} (M)$.
The so-obtained map $d$ satisfies $d \circ d =0$. 
Since $M$ is assumed to have no boundary, this definition is consistent with Stokes' formula:
$$\int _M d\omega \wedge \alpha = (-1)^{k+1} \int _M \omega \wedge d \alpha,$$
and gives in particular: $dT_\omega = T_{d\omega}$. 

More generally, suppose we are given $\ell \in \mathbf Z$, and a continuous linear operator $L : \Omega ^*(M) \to \Omega ^{*-\ell}(N)$ -- where $M$ and $N$ are orientable manifolds of dimension $D_M$ and $D_N$ respectively -- such that there is a continuous operator $\tilde L : \Omega _c ^{D_N -* +\ell} (N) \to \Omega _c ^{D_M - *}(M)$, with 
$$\int _N L(\omega) \wedge \alpha = \int _M \omega \wedge \tilde L (\alpha),$$
for every $\omega \in \Omega ^*(M)$ and $\alpha \in \Omega _c ^{D_N - * +\ell}(N)$. 
Then $L$ extends by continuity to $\mathcal D '^*(M)\to \mathcal D '^{*-\ell}(N)$, by setting $(L(T))(\alpha):=T(\tilde L(\alpha))$. 
This applies {\it e.g.\!} to inner products $\iota _\xi : \Omega ^*(M) \to \Omega ^{*-1}(M)$ by a vector field $\xi$ on $M$. One has $\tilde {\iota _\xi}=(-1)^{k+1} \iota _\xi $ on $\Omega _c ^{D - k + 1}(M)$, since $\iota _\xi$ is an anti-derivation (see {\it e.g.\!} \cite[Proposition 20.8]{Tu}). 

In local coordinates $(x_1, ..., x_D)$ on an open
subset $U \subset M$, every $k$-current $T \in \mathcal D'^k(U)$ can be written 
$T = \sum _{\vert I \vert =k} T_I dx_I$,
with $T _I \in \mathcal D'^0 (U)$. For every $\alpha \in \Omega _c ^{D-k}(U)$, 
one has $T(\alpha) =  \sum _{\vert I \vert =k} T_I (dx_I\wedge \alpha)$.

\subsection{De Rham $L^p$-cohomology: definitions and notation}
\label{sec-prelim}

We list and fix the definitions and notations for several objects that will appear repeatedly in the paper. 

Let $M$ be a $C^\infty$ orientable manifold (without boundary), henceforth endowed with a Riemannian metric.
We denote by $d\mathrm{vol}$ its Riemannian measure, and by $\vert v \vert$ the Riemannian length of a vector $v \in TM$. 

\begin{itemize}
\item Let $p\in (1, +\infty)$. The \textit{$L^p$-norm of $\omega \in \Omega^k (M)$} is
$$\Vert \omega \Vert _{L^p \Omega^k} = \bigl(\int _M \vert \omega \vert _m^p ~d\mathrm{vol}(m)\bigr)^{1/p},$$
where we set
$$\vert \omega \vert _m := \sup \{ \vert \omega (m; v_1, \dots, v_k) \vert: v_1, \dots, v_k \in T_m M, ~\vert v_i \vert = 1\}.$$

\item The space $L^p \Omega^k (M)$ is the norm completion of the normed space $\{\omega \in \Omega^k (M): \Vert \omega \Vert _{L^p \Omega^k} < +\infty \}$, \emph{i.e.}~ the Banach space of $k$-differential forms with measurable $L^p$ coefficients.

\item To every $\omega \in L^p \Omega^k(M)$, one associates the $k$-current $T_\omega$ defined by $T_\omega (\alpha):= \int _M \omega \wedge \alpha$.
The \textit{differential in the sense of currents} of $\omega \in L^p \Omega^k(M)$ is the $(k+1)$-current  $d\omega:= dT_\omega$.  
One says that $d\omega$ \textit{belongs to $L^p \Omega^{k+1}(M)$} if there exists $\theta \in L^p \Omega^{k+1}(M)$ such $d\omega = T_\theta$.
In this case we set $\Vert d\omega \Vert _{L^p \Omega^{k+1}}:= \Vert \theta \Vert _{L^p \Omega^{k+1}}$.

\item For $\omega \in \Omega^k (M)$, we set
$$\Vert \omega \Vert _{\Omega^{p,k}}:= \Vert \omega \Vert _{L^p \Omega^k} + \Vert d\omega \Vert _{L^p \Omega^{k+1}}.$$ 
The space $\Omega^{p,k} (M)$ is the norm completion of the normed space 
$\{\omega \in \Omega^k (M): \Vert \omega \Vert _{\Omega^{p,k}} < +\infty \}$. It is a Banach space that coincides with the subspace 
of $L^p \Omega^k (M)$ consisting 
of the $L^p$ $k$-forms whose differentials in the sense of currents belong to $L^p \Omega^{k+1} (M)$. 
Moreover the differential operator $d$ on $\Omega^{p,*} (M)$ agrees with the differential in the sense of currents. (See {\it e.g.}\! \cite[Lemma 1.5]{BR2} for a proof). 

\item The \textit{de Rham $L^p$-cohomology} of $M$ is the cohomology of the complex
$$\Omega^{p,0} (M) \stackrel{d_0}{\to} \Omega^{p,1} (M) \stackrel{d_1}{\to} \Omega^{p,2} (M) \stackrel{d_2}{\to} \dots$$ 
It is denoted by $L^p \mathrm{H_{dR}^*} (M)$. 
Its largest Hausdorff quotient is denoted by  $L^p \overline{\mathrm{H^*_{dR}}} (M)$ and is called the \textit{reduced de Rham $L^p$-cohomology} of $M$. 
The latter is a Banach space; its (quotient) norm is denoted by $\Vert \cdot \Vert _{L^p \overline{\mathrm{H^*}}}$.

\item Following Pansu, we  also define $\Psi^{p,k} (M)$ to be the space of $k$-currents $\psi \in \mathcal D '^k(M)$ that can be written
$\psi = \beta + d\gamma$, with $\beta \in L^p \Omega^k (M)$ and $\gamma \in L^p \Omega^{k-1} (M)$. 
In particular we have $\Psi^{p,0}(M) = L^p(M)$. 
Equipped with the norm
\begin{align*} 
\Vert \psi \Vert _{\Psi^{p,k}} := &\inf \Bigl\{\Vert \beta \Vert _{L^p \Omega^k} + \Vert \gamma \Vert _{L^p \Omega^{k-1}}: 
\psi = \beta + d\gamma, \\
& ~~\mathrm{with}~~ \beta \in L^p \Omega^k (M)~~\mathrm{and}~~ \gamma \in L^p \Omega^{k-1} (M)\Bigr\},
\end{align*}
the space $\Psi^{p,k} (M)$ is a Banach space, and the inclusion maps between differential complexes:
$$\Omega ^{p,*}(M) \subset \Psi^{p,*} (M) \subset \mathcal D '^*(M)$$
are continuous (see \cite[Lemma 1.3]{BR2} for a proof).

\item Suppose that $M$ carries a $C^\infty$ unit complete vector field $\xi$, and let $(\varphi _t)_{t \in{\mathbf R}}$ be its flow. 
Assume that $\varphi _t^*: L^p \Omega^k (M) \to L^p \Omega^k (M)$ is bounded for all $t \in \bf R$, $p\in (1, +\infty)$ and $k \in \bf N$. 
We set 
$$\ \ \ \Psi^{p,k} (M, \xi):= \{\psi \in \Psi^{p,k} (M): \varphi _t^* (\psi) = \psi~~ \mathrm{for~~ every}~~ t \in \bf R\}.$$
The differential complex $\Psi^{p,*} (M, \xi)$ is a closed subcomplex of $\Psi^{p,*}(M)$. Let 
$$\mathcal Z^{p,k} (M, \xi) := \Ker \bigl(d: \Psi^{p,k} (M, \xi) \to \Psi^{p,k+1} (M, \xi)\bigr)$$
be the space of $k$-cocycles.
\end{itemize}

\subsection{Poincar\'e duality} 
\label{ss - Poincaré duality}

Poincar\'e duality for de Rham $L^p$-cohomology takes the following form.

\begin{proposition}
\label{dR-Pd} 
Let $M$ be a complete oriented Riemannian manifold of dimension $D$.
Let $p \in (1, +\infty)$, $q = p/(p-1)$ be its H\"older conjugate, and $k \in \{0,\dots, D\}$. Then 
\begin{enumerate}
\item $L^p {\rm H}_{\mathrm{dR}}^k (M)$ is Hausdorff if and only if $L^q {\rm H}_{\mathrm{dR}}^{D-k +1} (M)$ is.
\item $L^p \overline{{\rm H}^k _{\mathrm{dR}}} (M)$ and $L^q \overline{{\rm H}^{D-k}_{\mathrm{dR}}} (M)$ are dual Banach spaces, via the perfect pairing 
$L^p \overline{{\rm H}^k _{\mathrm{dR}}} (M) \times L^q \overline{{\rm H}^{D-k}_{\mathrm{dR}}} (M) \to \mathbf R$, defined by  
$$([\omega _1], [\omega _2]) \mapsto \int _M \omega _1 \wedge \omega _2.$$
\end{enumerate}
\end{proposition}

\begin{proof} 
See \cite[Corollaire 14]{P1} or \cite{GT10}.
\end{proof}

The following terminology will be useful in the sequel.

\begin{definition}
\label{dR-definition}
Let $p,q \in (1, +\infty)$ and $k, \ell \in \{0, \dots,D\}$. The couples $(p,k)$ and $(q, \ell)$ are said to be \emph{Poincar\'e dual} if $p$ and $q$ are H\"older conjugate and if $\ell = D-k$.
\end{definition}

\section{Flows and $L^p$-cohomology}
\label{s - flows}

This section exploits some dynamical properties of flows acting on forms to extract information on $L^p$-cohomology.
The objects appearing in this section are defined in Section \ref{sec-prelim}.
In what follows, we keep $M$ a $C^\infty$ orientable manifold (without boundary) endowed with a Riemannian metric.

\subsection{Invariance, identification and vanishing}
 We review several results due to Pansu, see \cite[Proposition 10]{P1} or \cite[Section 1]{BR2}. 
 
Let $\xi$ be a $C^\infty$ unit complete vector field on $M$, and denote by $(\varphi _t)_{t \in{\mathbf R}}$ its flow. 
We assume that $\varphi _t^*: L^p \Omega^k (M) \to L^p \Omega^k (M)$ is bounded for all $t \in \bf R$, $p\in (1, +\infty)$ and $k \in \bf N$. This happens {\it e.g.\!} when $M$ is a manifold of {\it bounded geometry, i.e.\!} a manifold whose injectivity radius is bounded from below and whose sectional curvatures are bounded from above and from below.

\begin{proposition}
\label{dR-cohomologous}
Let $p \in (1, +\infty)$ and let $k \in {\bf N}$.
Then for every $\omega \in \Omega ^{p, k}(M) \cap \Ker d$ and $t \in \mathbf R$, the forms $\omega $ and $\varphi ^* _t \omega$ are cohomologous
in $L^p {\rm H}_{\mathrm{dR}}^k (M)$.
\end{proposition}

\begin{proof}
See {\it e.g.\!} \cite[Lemma 1.3]{BR2}.
\end{proof}

\begin{proposition}
\label{dR-identification} 
Let $p \in (1, +\infty)$ and $k \in {\bf N}^*$. 
Suppose that there exist $C, \eta >0$ such that for every $t \geqslant 0$, one has 
$$\Vert \varphi _t^* \Vert _{L^p \Omega^{k-1} \to L^p \Omega^{k-1}}
\leqslant C e^{-\eta t}.$$
Then: 
\begin{enumerate}
\item Let $\omega \in \Omega ^{p, k}(M) \cap \Ker d$. When $t \to +\infty$, the form $\varphi ^* _t \omega$ converges in the Banach space $\Psi ^{p,k} (M)$ (and so in the sense of currents); its limit $\omega _\infty$ is a closed current in $\mathcal Z ^{p,k} (M, \xi)$.

\item The map $\omega \mapsto \omega _\infty$ induces a canonical Banach isomorphism 
$$L^p \mathrm{H}_{\mathrm{dR}}^{k} (M) \simeq \mathcal Z^{p,k} (M, \xi).$$
In particular $L^p \mathrm{H}_{\mathrm{dR}}^{k} (M)$ is Hausdorff.
\end{enumerate}
\end{proposition}

\begin{proof} 
The statement is essentially contained in \cite[Proposition 10]{P1}. 
A proof also appears in \cite[Proposition 1.9]{BR2} under the stronger assumption that $\max _{i=k-2, k-1} \Vert \varphi _t ^* \Vert _{L^p \Omega ^i \to L^p \Omega ^i} \leqslant C e^{-\eta t}$. 
The extra assumption served only in parts (3) and (4) of the proof, to show that $\displaystyle \lim _{t \to +\infty} \Vert \varphi _t ^* (d\theta) \Vert _{\Psi ^{p,k}}= 0$ for every $\theta \in L^p \Omega^{k-1}(M)$. 
But the weaker hypothesis $\Vert \varphi _t^* \Vert _{L^p \Omega^{k-1} \to L^p \Omega^{k-1}} \leqslant C e^{-\eta t}$ is enough to prove this property.
Indeed, by combining the definition of $\Vert \cdot \Vert_{\Psi ^{p,k}}$ with this assumption, one has  
$$\Vert \varphi _t ^* (d\theta) \Vert _{\Psi ^{p,k}} =  \Vert d (\varphi _t ^* \theta) \Vert _{\Psi ^{p,k}} \leqslant \Vert \varphi _t ^* \theta \Vert _{L^p \Omega ^{k-1}} \to 0$$
when $t \to \infty$.
 \end{proof}

\begin{corollary}
\label{dR-vanishing}
Let $p \in (1, +\infty)$ and $k \in {\bf N}^*$. 
Suppose that there exist $C, \eta >0$ such that for every $t \geqslant 0$, one has 
$$\max _{i= k-1, k} \Vert \varphi _t ^* \Vert _{L^p \Omega ^{i} \to L^p \Omega ^{i}} 
\leqslant C e^{-\eta t}.$$
Then $L^p \mathrm{H}^{k} _{\mathrm{dR}}(M)= \{0\}$.
\end{corollary}

\begin{proof} 
Our assumption implies that
$\Vert \varphi _t ^* \Vert _{\Psi^{p, k} \to \Psi^{p, k}} \leqslant C e^{-\eta t}$; and also that $L^p \mathrm{H}^{k} _{\mathrm{dR}}(M) \simeq \mathcal Z^{p,k}(M, \xi)$ by Proposition \ref{dR-identification}. Since the elements of $\mathcal Z^{p,k}(M, \xi)$ are $\varphi _t$-invariant, one gets that $\mathcal Z^{p,k}(M, \xi) = \{0\}$. 
Therefore $L^p \mathrm{H}^{k} _{\mathrm{dR}}(M)= \{0\}$.
\end{proof}

\subsection{Boundary values, Poincar\'e duality revisited}
\label{ss - bv}

In this section, the oriented Riemanniann manifold $M$ is supposed to be complete. We assume futhermore that $M$ and the unit vector field $\xi$ are such that the pair $(M, \xi)$ is $C^\infty$-diffeomorphic to a pair of the form $(\mathbf R \times N, \frac{\partial}{\partial t})$, where the vector field $\frac{\partial}{\partial t}$ is carried by the $\mathbf R$-factor. 

We think of $N$ as a ``boundary'' of $M$. Under some dynanical assumptions, we will represent the spaces $L^p \mathrm{H}_{\mathrm{dR}}^{k} (M)$ and the Poincar\'e duality on the boundary $N$ (see Proposition \ref{boundary-prop} and Theorem \ref{boundary-thm} below).

Let $\pi : M \to N$ be the projection map and $\pi^*: \mathcal D'^i(N) \to \mathcal D'^i(M)$ be the continuous extension of the pull-back map $\pi ^* : \Omega ^i(N) \to \Omega ^i (M)$.
We set $n = :\dim N$ so that $D := \dim M = n+1$.

\begin{proposition}
\label{boundary-prop} 
Let $p \in (1, +\infty)$ and $k \in \mathbf N ^*$. 
Suppose that there exist $C, \eta >0$ such that for $t \geqslant 0$:
$$\Vert \varphi _t ^* \Vert _{L^p \Omega ^{k-1} \to L^p \Omega ^{k-1}} 
\leqslant C e^{-\eta t}.$$
Then for every $\psi \in \mathcal Z ^{p,k} (M, \xi)$, there exists $T \in \mathcal D '^{k}(N) \cap \Ker d$ such that $\psi = \pi ^* (T)$.
\end{proposition}

\begin{proof}
Recall that $\xi = \frac{\partial}{\partial t}$ and that the flow of $\xi$ is denoted by $\varphi _t$. 
Every $\psi \in \mathcal Z ^{p,k} (M, \xi)$ is $\varphi_t$-invariant; therefore showing that $\psi = \pi ^* (T)$ is equivalent to proving that $\iota _\xi \psi = 0$. From  
Proposition \ref{dR-identification}, there exists $\omega \in \Omega ^{p, k}(M) \cap \Ker d$ such that $\psi = \lim _{t \to +\infty} \varphi ^* _t (\omega)$
in the sense of currents. 
Since the map $\iota _\xi : \mathcal D '^{k}(M) \to \mathcal D '^{k-1}(M)$ is continuous, one obtains that  
$\iota _\xi \psi = \lim _{t \to +\infty} \varphi ^* _t (\iota _\xi \omega)$ in the sense of currents. But $\iota _\xi \omega \in L^p \Omega ^{k-1}(M)$, and by assumption one has
$\Vert \varphi _t ^* \Vert _{L^p \Omega ^{k-1} \to L^p \Omega ^{k-1}} \to 0$ when $t \to +\infty$. Thus $\iota _\xi \psi = 0$.

Lastly, since $d\psi =0$, one gets that $\pi ^* (dT) =0$, which in turn implies that  $dT=0$. Thus $T \in \mathcal D '^{k}(N) \cap \Ker d$, as expected.
\end{proof}

Let $\chi$ be a non-negative $C^\infty$ function on $M$, depending only on the $\mathbf R$-variable, such that $\chi (t) =0$ for $t \leqslant 0$ and $\chi (t) =1$ for $t \geqslant 1$.

\begin{theorem} 
\label{boundary-thm} 
Let $p, q \in (1, +\infty)$ and $k, \ell \in \{1, \dots, n\}$ be such that $(p,k)$ and $(q, \ell)$ are Poincar\'e dual --- see Definition \ref{dR-definition}. Suppose that there exist $C, \eta >0$ such that for $t \geqslant 0$:
\begin{enumerate}
\item $\Vert \varphi _t ^* \Vert _{L^p \Omega ^{k-1} \to L^p \Omega ^{k-1}} \leqslant C e^{-\eta t}$,
\item $\Vert \varphi _{-t} ^* \Vert _{\Ker \iota _\xi \cap L^q \Omega^{\ell} \to \Ker \iota _\xi \cap L^q \Omega^{\ell}}  \leqslant C e^{-\eta t}$.
\end{enumerate}
Then for every $\theta  \in \Omega _c ^{\ell-1}(N)$, the form $d(\chi \cdot \pi ^* \theta)$ belongs to the space $\Omega ^{q,\ell}(M) \cap \Ker d$; and for every
$\omega \in \Omega ^{p, k}(M) \cap \Ker d$, one has:
$$\int _M \omega \wedge d(\chi \cdot \pi ^* \theta) = T(\theta),$$
where $T$ is the closed $k$-current on $N$ such that 
$$\omega_\infty = \lim _{t \to +\infty} \varphi _t ^* (\omega) = \pi ^* (T),$$ as in Propositions \ref{dR-identification} and \ref{boundary-prop}.
\end{theorem}

As a consequence of Theorem \ref{boundary-thm}, 
we will prove:

\begin{corollary}
\label{boundary-cor}Suppose that the assumptions of Theorem \ref{boundary-thm} are satisfied.
Then the classes of the $d(\chi \cdot \pi ^* \theta)$'s 
(where $\theta \in \Omega _c ^{\ell-1}(N)$) form a dense subspace in $L^q \overline{\mathrm H _{\mathrm{dR}} ^{\ell}} (M)$. 
Moreover when $\theta = d \alpha $ is an {\it exact} form, with $\alpha \in \Omega _c ^{\ell-2}(N)$, then $[d(\chi \cdot \pi ^* \theta)] =0$ in $L^q \overline{\mathrm H _{\mathrm{dR}} ^{\ell}} (M)$.
\end{corollary}

Recall that $L^p \overline{{\rm H}^{k} _{\mathrm{dR}}} (M)$ and $L^q \overline{{\rm H}^{\ell}_{\mathrm{dR}}} (M)$ are dual Banach spaces, via the pairing 
$([\omega _1], [\omega _2]) \mapsto \int _M \omega _1 \wedge \omega _2$ (see Proposition \ref{dR-Pd}).
In combination with Theorem \ref{boundary-thm} and Corollary \ref{boundary-cor} above, this yields immediately to the:

\begin{corollary}
\label{boundary-norm}
Suppose that the assumptions of Theorem \ref{boundary-thm} are satisfied. Let $\omega \in \Omega ^{p, k}(M) \cap \Ker d$ and let $T \in  \mathcal D '^{k}(N) \cap \Ker d$ be such that $\displaystyle \lim _{t \to +\infty} \varphi _t ^* (\omega) = \pi ^* (T)$. 
Then the norm of $[\omega ]$ in $L^p {\rm H}^{k} _{\mathrm{dR}} (M)$ satisfies:
$$\bigl\Vert [\omega ] \bigr\Vert_{L^p {\mathrm{H^k}}} = \sup \Bigl\{T(\theta) : \theta \in \Omega _c ^{\ell-1}(N),~ \bigl\Vert [d(\chi \cdot \pi ^* \theta) ] \bigr\Vert_{L^q \overline{\mathrm{H^\ell}}} \leqslant 1 \Bigr\}.$$
\end{corollary}

\begin{proof}[Proof of Theorem \ref{boundary-thm}] 

{\it Step 1}. We first show that the $L^q$-norm of the form $d(\chi \cdot \pi ^* \theta)$ is finite. 
Set $\alpha := \pi ^* \theta$ for simplicity. One has 
$$d(\chi \cdot \alpha) = d\chi \wedge \alpha  + \chi \cdot d \alpha.$$
The form $d\alpha$ belongs to $\Ker \iota _\xi$ and is $\varphi _t$-invariant. 
With the assumption (2) we obtain (since $s \geqslant 0$):
\begin{align*}
\Vert \chi \cdot d\alpha \Vert _{L^q \Omega^{\ell}} 
&\leqslant \Vert \mathbf{1}_{t\geqslant 0}\cdot d\alpha \ \Vert _{L^q \Omega^{\ell}}\\
&= \sum _{i=0}^\infty  \Vert \mathbf{1}_{t\in [i, i+1]} \cdot d\alpha  \Vert _{L^q \Omega^{\ell}}\\
&\leqslant C \sum _{i=0}^\infty e^{-\eta i}\Vert \mathbf{1}_{t\in [0,1]} \cdot d\alpha \Vert _{L^q \Omega^{\ell}}\\
&= \frac{C}{1 - e^{-\eta}} \Vert \mathbf{1}_{t\in [0,1]} \cdot d\alpha \Vert _{L^q \Omega^{\ell}}.
\end{align*}
which is finite since $\mathbf{1}_{t\in [0,1]} \cdot d\alpha$ has compact support.

It remains to bound from above the $L^q$-norm of $d\chi \wedge \alpha$. One has 
$d \chi  = \chi'(t)dt$, with $\chi'$ supported on $[0,1]$. Thus
$$\Vert d\chi \wedge \alpha \Vert_{L^q \Omega^{\ell}} \leqslant \Vert \chi' \cdot \alpha \Vert _{L^q \Omega^{\ell-1}} 
\leqslant C_1 \Vert {\bf1}_{t \in [0,1]} \cdot \alpha \Vert _{L^q \Omega^{\ell-1}},$$
with $C_1 = \Vert \chi' \Vert _\infty$.
Since ${\bf1}_{t \in [0,1]} \cdot \alpha$
has compact support, the $L^q$-norm of $d\chi \wedge \alpha$ is finite too. The statement follows.

{\it Step 2}. We now compute $\int \omega \wedge d(\chi \cdot \alpha)$. Since  
$\varphi _t ^*(\omega)$ and  $\omega$
are cohomologous (by Proposition \ref{dR-cohomologous}), one has thanks to Proposition \ref{dR-Pd}(2):
\begin{align*}
\int \omega \wedge d(\chi \cdot \alpha) &= \int \varphi _t^*(\omega) \wedge d(\chi \cdot \alpha)\\
&= \int \varphi _t^*(\omega) \wedge d\chi \wedge \alpha  + \int \varphi _t^*(\omega) \wedge (\chi \cdot d \alpha).
\end{align*}
Since the form $d\chi \wedge \alpha$ belongs to $\Omega _c ^{\ell}(M)$, one has 
$$\lim _{t \to \infty} \int \varphi _t^*(\omega) \wedge d\chi \wedge \alpha = (\pi ^*T)(d\chi \wedge \alpha),$$
indeed $\varphi _t^*(\omega)$ tends to $\pi ^*T$ in the sense of currents thanks to assumption (1), Propositions \ref{dR-identification} and \ref{boundary-prop}.

One observes that the map $\pi^* : \mathcal D '^{i}(N) \to \mathcal D '^{i}(M)$ can be written as $(\pi ^* T)(\beta) = T\bigl(j(\beta)\bigr)$
where $j: \Omega _c ^{D-i}(M) \to \Omega _c ^{D-1-i}(N)$ is defined by
$$j(\beta) = \int _{\mathbf R} (\iota _{\xi} \beta) _{(t, \cdot)}~ dt$$
(we recall that $\xi = \frac{\partial}{\partial t}$).
Since the inner product is an anti-derivation (see {\it e.g.\!} \cite[Proposition 20.8]{Tu}) and since $\iota _\xi \alpha =0$, 
one has 
$$\iota _\xi (d\chi \wedge \alpha) = (\iota _\xi d\chi)\wedge \alpha - d\chi \wedge (\iota _\xi \alpha) = \chi ' \cdot  \pi ^* \theta.$$
Therefore $j(d\chi \wedge \alpha) = \int _{\mathbf R} \chi '(t) \cdot \theta ~dt = \theta$, and we obtain 
$$(\pi ^*T)(d\chi \wedge \alpha) = T(\theta).$$

{\it Step 3}. According to the previous paragraph, it remains to prove that 
$$\lim _{t \to +\infty} \int \varphi _t^*(\omega) \wedge (\chi \cdot d \alpha) =0.$$
For $s>0$, let $\chi_s : M \to \mathbf R$ be a $C^\infty$-function depending only on the $\mathbf R$-variable, such that $\chi _s (t) = \chi (t)$ for $t \leqslant s$ and $\chi_s (t) =0$ for $t \geqslant s+1$.
Observe that $\chi_s \cdot d \alpha$ is $C^\infty$ with compact support. 
We claim that:
\begin{itemize}
\item For every $s>0$, one has $\lim _{t \to +\infty} \int \varphi _t^*(\omega) \wedge (\chi_s \cdot d \alpha) =0$,
\item $\int \varphi _t^*(\omega) \wedge \bigl((\chi - \chi_s) \cdot d \alpha\bigr)$ tends to 0 uniformly in $t>0$ when $s \to +\infty$.
\end{itemize}
As explained above, the claim completes the proof of the theorem.
The first item of the claim follows from the same type of argument that we used in Step 2. Note that here we have $\iota _\xi (\chi_s \cdot d \alpha)= \chi _s \cdot \iota_\xi \pi ^* \theta = 0$.

To prove the second item, recall from Proposition \ref{dR-identification} that $\varphi _t^*(\omega)$ converges in $\Psi ^{p, k}(M)$ when $t \to +\infty$.
Therefore there exists $M>0$ such that $\Vert \varphi _t^*(\omega) \Vert _{\Psi ^{p, k}} \leqslant M$ for every $t>0$.
Write $\varphi _t^*(\omega) = \beta _t + d \gamma _t$ with $\Vert \beta_t \Vert _{L^p \Omega ^{k}} + \Vert \gamma_t \Vert _{L^p \Omega ^{k-1}} \leqslant 2M$.
Observe that $(\chi - \chi_s) \cdot d \alpha$ belongs to $\Omega ^{q, \ell}(M)$. Since $M$ is complete, the space $\Omega^\ell _c (M)$ is dense in $\Omega ^{q, \ell}(M)$ (see \cite[Proof of Lemma 4]{GT10}). Thus for every $t>0$, one gets with H\"older:
\begin{align*}
\bigl\vert &\int \varphi _t^*(\omega) \wedge \bigl( (\chi - \chi_s) \cdot d \alpha\bigr) \bigr\vert \\ 
=  \bigl\vert &\int \beta_t \wedge \bigl( (\chi - \chi_s) \cdot d \alpha \bigr) + (-1)^k \int \gamma_t \wedge d(\chi - \chi_s) \wedge d \alpha \bigr\vert \\
\leqslant \ \  ~& 2M \Vert (\chi - \chi_s) \cdot d \alpha \Vert _{L^q \Omega ^{\ell}} + 2M \Vert d(\chi - \chi_s) \wedge d \alpha \Vert _{L^q \Omega^{\ell+1}}.
\end{align*}
By the same type of argument that we used in Step 1, one obtains that the last two norms tend to  $0$ when $s \to +\infty$.
\end{proof}

\begin{proof}[Proof of Corollary \ref{boundary-cor}]
Let  $\omega \in \Omega ^{p, k}(M)\cap \Ker d$ be such that 
$$\int _M \omega \wedge d(\chi \cdot \pi^* \theta) = 0$$ 
for every $\theta 
\in \Omega _c ^{\ell-1}(N)$. According to Poincar\'e duality (Proposition \ref{dR-Pd}), it is enough to show that $[w] =0$ in $L^p \mathrm H _{\mathrm {dR}} ^{k} (M) $. By Propositions \ref{dR-identification} and \ref{boundary-prop}, this is equivalent
to $T =0$, where $T \in D '^{k}(N) \cap \Ker d$ is the $k$-current so that $\omega _\infty = \pi ^*(T)$. From Theorem \ref{boundary-thm} and our assumption, one has for every $\theta \in \Omega _c ^{\ell-1}(N)$: 
$$T(\theta) = \int _M \omega \wedge d(\chi \cdot \pi^* \theta) = 0.$$
Thus $T =0$.

Suppose now that $\theta = d \alpha $ is an exact form, with $\alpha \in \Omega _c ^{\ell-2}(N)$. Then by using again Poincar\'e duality as above, we obtain that the class of $d(\chi \cdot \pi^* \theta)$ is null in $L^q \overline{\mathrm H _{\mathrm{dR}} ^{\ell}} (M)$, since $T(\theta) = dT (\alpha) = 0$ for every $T \in D '^{k}(N) \cap \Ker d$.
\end{proof}

\section{The Lie group case}
\label{s - Lie}

We consider in this section a connected Lie group $G = \mathbf R \ltimes_\delta \!H$, whose law is $(t,x)\cdot (s, y) = (t+s, x e^{t\delta}(y))$, where $\delta \in \mathrm{Der}(\frak h)$ is a derivation of the Lie algebra $\frak h$ of the closed subgroup $H$. {\bf We will always assume that the eigenvalues of $\delta$ all have non-positive real parts, and that $\mathrm{trace}(\delta) <0$}. In particular $G$ is non-unimodular. We set $n := \dim H$ so that $D :=\dim G = n+1$. Equip $G$ with a left-invariant Riemannian metric and with the associated Riemannian measure $d\mathrm{vol}$.

\subsection{A strip decomposition}
We exhibit some regions of the set of parameters $(p, k) \in (1, +\infty) \times \{1, \dots, n\}$, where the results of the previous sections apply and give some informations on the spaces $L^p \mathrm{H}^{k} _{\mathrm{dR}}(G)$ --- see Theorem \ref{Lie-thm} below. These regions form a kind of a ``strip decomposition'' of the set of parameters. Examples will be given in Sections \ref{s - real hyp sp} and \ref{s - complex hyp sp}.

We start with the following lemma which translates the norm assumptions that appeared repeatedly in the previous sections, into simple inequalities between the exponent $p$ and the eigenvalues of $-\delta$. 

Let  $0 \leqslant \lambda _1 \leqslant \lambda _2 \leqslant \dots \leqslant \lambda _n$  be the ordered list of the real parts of the eigenvalues of $-\delta$, 
enumerated with their multiplicities in the generalized eigenspaces. We denote by $w_k = \sum _{i=1}^k \lambda _i$ 
the sum of the $k$ first real part eigenvalues, and by $W_k = \sum _{j=0}^{k-1} \lambda _{n-j}$ 
the sum of the $k$ last ones. 
We also set $w_0 = W_0 = 0$.

Note that we always have: $w_{k-1} \leqslant w_k \leqslant W_k$ and $w_{k-1} \leqslant W_{k-1} \leqslant W_k$, but the comparison between $w_k$ and $W_{k-1}$  is not automatic. 
This can be seen for instance by considering the example where $\lambda_1 = \lambda_2 = \dots = \lambda_{n-1} = 1$ and $\lambda_n = a \geqslant 1$; then for $a>2$ we have $W_{k-1} > w_k$, for $a=2$ we have $W_{k-1} = w_k$ and for $a<2$ we have $W_{k-1} < w_k$. 

Let $h = \sum _{i=1}^n \lambda _i >0$ be the trace of $-\delta$. If $w_k =0$ (resp. $W_k =0$), we put $\frac{h}{w_k} := +\infty$ (resp. $\frac{h}{W_k} := +\infty$). One has $w_k + W_{n-k} =h$ for every $k \in \{0, \dots, n\}$; therefore $\frac{h}{w_k}$ 
and $\frac{h}{W_{n-k}}$ are H\"older conjugated (even if $w_k$ or $W_{n-k}$ is $0$).

\begin{lemma} 
\label{Lie-equivalence} 
Let $\xi = \frac{\partial}{\partial t}$ be the left-invariant vector field on $G$ 
carried by the $\mathbf R$-factor, and let $\varphi _t$ be its flow (it is just a translation along the $\mathbf R$-factor). Let $p, q \in (1, +\infty)$ and $k, \ell \in \{1, \dots,n\}$ be such that $(p, k)$ and $(q, \ell)$ are Poincar\'e dual. 
The following properties are equivalent:
\begin{enumerate}
 \item There exist $C, \eta >0$ such that for $t \geqslant 0$:
$$\Vert \varphi _t ^* \Vert _{L^p \Omega ^{k-1} \to L^p \Omega ^{k-1}} 
\leqslant C e^{-\eta t}.$$
\item There exist $C, \eta >0$ such that for $t \geqslant 0$:
$$\Vert \varphi _{-t} ^* \Vert _{\Ker \iota _\xi \cap L^q \Omega ^{\ell} \to L^q \Omega ^{\ell} \cap \Ker \iota _\xi} 
\leqslant C e^{-\eta t}.$$
\item We have: $p<\frac{h}{W_{k-1}}.$
\item We have: $q > \frac{h}{w_{\ell}}$.
\end{enumerate}
In particular conditions (1) and (2) in Theorem \ref{boundary-thm} are
equivalent when $M= G$.
\end{lemma}
\begin{proof} The equivalences (1)$\Leftrightarrow$(3) and (2)$\Leftrightarrow$(4) follow from the same line of arguments as in \cite[Proof of Proposition 2.1]{BR2}. 
 To obtain (3)$\Leftrightarrow$(4) one notices that $\frac{h}{W_{k-1}}$ and
 $\frac{h}{w_{\ell}}$ are H\"older conjugated, since $w_{D-k} + W_{k-1} = h$.
\end{proof}
We can now summarize and specify the results of the previous sections, to obtain the following statement that complements results of Pansu \cite[Corollaire 53 and Proposition 57]{Pa99}.

\begin{theorem} 
\label{Lie-thm}
Let $G$ be a Lie group as above. Let $p \in (1, +\infty)$ and $k \in \{1, \dots,n\}$.
\begin{enumerate}
 \item{}{\rm [Vanishing]}~If $p<\frac{h}{W_k}$ or $p>\frac{h}{w_{k-1}}$, then $L^p \mathrm{H}^{k} _{\mathrm{dR}}(G) =\{0\}$.
 \item{}{\rm [Hausdorff property]}~If $\frac{h}{W_k} <p<\frac{h}{W_{k-1}}$, then $L^p \mathrm{H}^{k} _{\mathrm{dR}}(G)$ is Hausdorff and Banach-isomorphic to $\mathcal Z^{p,k}(G, \xi)$.
 \item{}{\rm [Density]}~If $\frac{h}{w_k} <p<\frac{h}{w_{k-1}}$, then the classes of the $d(\chi \cdot \pi ^* \theta)$'s 
(where $\theta \in \Omega _c ^{k-1}(H)$) form a dense subspace in $L^p \overline{\mathrm H _{\mathrm {dR}} ^{k}}(G)$.
 \item{}{\rm [Poincar\'e duality (on the boundary)]}~Let $(q,\ell)$ be the Poincar\'e dual of $(p,k)$.
Then we have $\frac{h}{W_k} <p<\frac{h}{W_{k-1}}$ if and only if
$\frac{h}{w_\ell} <q<\frac{h}{w_{\ell-1}}$, in which case for every $[\omega] \in L^p \mathrm{H}^{k} _{\mathrm{dR}}(G)$ and every $[d(\chi \cdot \pi ^* \theta)] \in  L^q \overline{\mathrm H _{\mathrm {dR}} ^{\ell}}(G)$, we have 
$$\int _G \omega \wedge d(\chi \cdot \pi^* \theta) = T(\theta),$$
where $T$ is the closed $k$-current on $H$ such that $\displaystyle \lim _{t \to +_\infty} \varphi_t ^*(\omega) = \pi ^*(T)$ (as in Propositions \ref{dR-identification} and \ref{boundary-prop}). Moreover, one has:
$$~~~~~~~\bigl\Vert [\omega ] \bigr\Vert_{L^p {\mathrm{H^k}}} = \sup \Bigl\{T(\theta) : \theta \in \Omega _c ^{\ell-1}(H),~ \bigl\Vert [d(\chi \cdot \pi ^* \theta) ] \bigr\Vert_{L^q \overline{\mathrm{H^\ell}}} \leqslant 1 \Bigr\}.$$
\end{enumerate}
\end{theorem}

\begin{proof}
Item (2) follows from Proposition \ref{dR-identification} and Lemma \ref{Lie-equivalence}. 
Item (3) is a consequence of Corollary \ref{boundary-cor} and Lemma \ref{Lie-equivalence}. 
One deduces Item (4) from Theorem \ref{boundary-thm}, Corollary \ref{boundary-norm} and Lemma \ref{Lie-equivalence}. 

It remains to prove Item (1). 
Suppose first that $p<\frac{h}{W_k}$. Since $\frac{h}{W_k} \leqslant \frac{h}{W_{k-1}}$, Lemma \ref{Lie-equivalence} implies that $\max _{i= k-1, k} \Vert \varphi _t ^* \Vert _{L^p \Omega ^{i} \to L^p \Omega ^{i}} \leqslant C e^{-\eta t}$. 
Thus by Corollary \ref{dR-vanishing}, one has $L^p \mathrm{H}^{k} _{\mathrm{dR}}(G)= \{0\}$, and the first part of Item (1) is proved. 
The second part follows from the first one, in combination with Poincar\'e duality (Proposition \ref{dR-Pd}), and the fact that $L^q \mathrm H _{\mathrm{dR}}^{D-k+1}(G)$ is Hausdorff follows from Lemma \ref{Lie-equivalence} and Proposition \ref{dR-identification}.
\end{proof}

\begin{remark} In the special case where $H= \mathbf{R}^n$, Pansu \cite[Proposition 27]{P1} has complemented the picture seen in Theorem \ref{Lie-thm}, by showing that the {\it torsion} in $L^p \mathrm{H}^{k} _{\mathrm{dR}}(G)$ -- {\it i.e.\!} the kernel of the quotient map $L^p \mathrm{H}^{k} _{\mathrm{dR}}(G) \to  L^p \overline{\mathrm H _{\mathrm {dR}} ^{k}}(G)$ -- is non-zero for $\frac{h}{W_{k-1}} <p<\frac{h}{w_{k-1}}$ (note that  there is no such $p$ for real hyperbolic spaces). 
\end{remark}

\subsection{Norm estimates}

We complement the norm expression obtained in Theorem \ref{Lie-thm}(4). The following inequalities are not optimal, however they are often sufficient for our purposes.

\begin{proposition}
\label{Lie-inequalities}
Let $\ell \in \{1, \dots, n\}$ and $q>\frac{h}{w_{\ell}}$. There exists a constant $C>0$ such that  for every $\theta \in \Omega _c ^{\ell-1}(H)$, the norm of the class of 
$d(\chi \cdot \pi ^* \theta)$ in $L^q \overline{\mathrm H _{\mathrm {dR}} ^{\ell}} (G)$ satisfies
$$\bigl\Vert [d(\chi \cdot \pi ^* \theta)] \bigr\Vert_{L^q \overline{\mathrm{H^\ell}}(G)} 
\leqslant C \inf_{t \in \mathbf R} \bigl\{\Vert (e^{t\delta})^* d\theta \Vert _{L^q \Omega ^{\ell}(H)} + \Vert (e^{t\delta})^* \theta \Vert _{L^q \Omega^{\ell-1}(H)} \bigr\}.$$
 \end{proposition}

 \begin{proof} Since $q>\frac{h}{w_{\ell}}$, Lemma \ref{Lie-equivalence} shows that the assumptions of Theorem \ref{boundary-thm} are satisfied.
By analysing Step 1 in the proof of Theorem \ref{boundary-thm}, and by using the homogeneity  of $G$, one sees that there exists a constant $C>0$ such that for every $\theta \in \Omega _c ^{\ell-1}(H)$:
\begin{equation}\label{Lie-eqn}
\Vert d(\chi \cdot \pi ^* \theta) \Vert _{\Omega ^{q, \ell}(G)} \leqslant C \bigl(\Vert d\theta \Vert _{L^q \Omega ^{\ell}(H)} + \Vert \theta \Vert _{L^q \Omega ^{\ell-1}(H)} \bigr).
\end{equation}
The left translation by $(t,1_H)$, which we denote by $L_{(t,1_H)}$, is an isometry of $G$. Therefore it acts by isometry on $L^q \overline{\mathrm H _{\mathrm{dR}} ^{\ell}} (G)$. 
One has:
$$L_{(t,1_H)}^*\bigl(d(\chi \cdot \pi ^* \theta)\bigr) = 
d (\chi \circ \varphi_t \cdot \pi ^* (e^{t\delta})^* \theta)
=\varphi _t ^*\bigl(d(\chi \cdot \pi ^* (e^{t\delta})^* \theta)\bigr).$$
Thus, by Proposition \ref{dR-cohomologous}, the forms
$L_{(t,1_H)}^*(d(\chi \cdot \pi ^* \theta))$ and $d(\chi \cdot \pi ^* (e^{t\delta})^* \theta)$ are cohomologous in $L^q \overline{\mathrm H _{\mathrm{dR}} ^{\ell}} (G)$. So the classes of $d(\chi \cdot \pi ^* \theta)$ and of $d(\chi \cdot \pi ^* (e^{t\delta})^* \theta)$ have equal norm.
One obtains the proposition by applying inequality (\ref{Lie-eqn}) to the $(e^{t\delta})^* \theta$'s.
\end{proof}
Proposition \ref{Lie-inequalities} provides upper bounds for norms of classes by means of norms of forms. These upper bounds on norms of forms can themself be obtained thanks to the following lemma:

\begin{lemma}
\label{Lie-lemma}
Let $H$ be a connected Lie group equipped with a left-invariant Riemannian metric, and let $\frak h$  be its Lie algebra. Let $\delta \in \mathrm{Der}(\frak h)$ be an $\mathbf R$-diagonalizable derivation of $\frak h$. Then for every $k \in \mathbf N$ the endomorphism $\delta ^* : \Lambda ^k \frak h ^* \to \Lambda ^k \frak h ^* $ is diagonalizable too. Let $\{\omega _I\} \subset \Lambda ^k \frak h ^*$ be a basis of eigenvectors, and denote by $\mu_I \in \mathbf R$ the corresponding eigenvalues. By identifying 
$\Lambda ^k \frak h ^* $ with the space of left-invariant $k$-forms on $H$,
every $\omega \in \Omega ^k(H)$ decomposes uniquely as $\omega = \sum _I f_I \omega_I$, 
where $f_I \in \Omega ^0(H)$. One has
$$\Vert (e^\delta)^* \omega \Vert
_{L^p \Omega ^k(H)} \asymp _D \sum _I e^{\mu _I -\frac{h}{p}}\Vert f_I \Vert _{L^p(H)},$$
where $h$ is the trace of $\delta$, and $D>0$ is a constant which depends only on $p$ and the choice of $\{\omega _I\}$.
\end{lemma}
\begin{proof}
 Since the norms on $\Lambda^k \mathfrak h^*$ are all equivalent, there exists a constant $C>0$ such that for every 
$\omega = \sum _I f_I \omega_I \in \Omega^k(H)$ and $g \in H$: 
$$ \vert \omega \vert _g \asymp_C 
\bigl( \sum _I \vert f_I(g) \vert^p \bigr) ^{\frac{1}{p}}.$$
On the other hand: 
$$(e^\delta)^* \omega  = \sum _I (f_I \circ e^\delta) \cdot (e^\delta)^* \omega _I 
= \sum _I (f_I \circ e^\delta)  \cdot e^{\mu_I }\omega _I .$$
Therefore: 
\begin{align*}
\Vert (e^\delta)^* \omega \Vert _{L^p \Omega^k}^p &= \int _H \vert {e^\delta}^* \omega \vert _g^p ~d\mathrm{vol}(g)\\
&\asymp _{C^p} \int_H \sum _I \vert e^{\mu_I }(f_I \circ e^\delta)(g) \vert^p ~d \mathrm{vol}(g)\\
&= \int_H \sum _I  e^{p\mu_I}\vert f_I (g) \vert^p \mathrm{Jac}( e^{-\delta}) (g) ~d \mathrm{vol} (g)\\
&= \int_H \sum _I  e^{p(\mu_I -\frac{h}{p})} \vert f_I (g) \vert^p ~d \mathrm{vol} (g)\\
&= \sum _I e^{p(\mu _I -\frac{h}{p})}\Vert f_I \Vert _{L^p}^p,
\end{align*}
since the Jacobian of $e^{\delta }$ is $e^{h}$. Thus:
$$\Vert (e^\delta)^* \omega \Vert _{L^p \Omega^k} \asymp _C 
\bigl(\sum _I e^{p(\mu _I -\frac{h}{p})}\Vert f_I \Vert _{L^p}^p\bigl) ^{\frac{1}{p}} \asymp _D \sum _I e^{\mu _I -\frac{h}{p}}\Vert f_I \Vert _{L^p},$$
where $D$ depends only on $p$ and $\{\omega _I\}$.
\end{proof}

\section{Real hyperbolic spaces} 
\label{s - real hyp sp}

We collect applications to a first series of concrete examples, namely real hyperbolic spaces. 
Let $R = {\mathbf R} \ltimes_\delta \!{\mathbf R}^{n}$ with $\delta = - \mathrm{id}_{\mathbf R^n} \in \mathrm{Der}(\mathbf R ^n)$. 
Then $R$ is a solvable Lie group isometric to the real hyperbolic space $\mathbb H^{n+1} _{\mathbf R}$. 
Its cohomology admits the following rather simple description, which appears already in \cite{P1} (apart from the density statement).

\begin{theorem}
\label{real-thm}
For every $k\in \{1, \dots , n\}$, one has: 
\begin{enumerate}
\item $L^p \mathrm{H}^k _{\mathrm{dR}}(R) =\{0\}$ for $1< p<\frac{n}{k}$ or $p>\frac{n}{k-1}$. 
\item If $\frac{n}{k}<p <\frac{n}{k-1}$, then $L^p \mathrm{H}^{k} _{\mathrm{dR}}(R)$ is Hausdorff, and Banach isomorphic to  $\mathcal Z^{p,k}(R, \xi)$.  
The space $\{\pi ^*d\theta ~\vert~ \theta \in \Omega _c ^{k-1}(\mathbf R ^n)\}$ is dense in $\mathcal Z^{p,k}(R, \xi)$; in particular $L^p \mathrm{H}^{k} _{\mathrm{dR}}(R)$ is non-zero. 
\end{enumerate}
\end{theorem}

\begin{proof}
For $k\in \{1, \dots , n\}$ one has $w_k = W_k = k$ and $h=n$. Item (1) comes from Theorem \ref{Lie-thm}(1). 
Item (2) is an application of Theorem \ref{Lie-thm}(2) and (3), in combination with Proposition \ref{dR-identification}. Indeed we have: 
$$\lim _{t \to +\infty} \varphi ^* _t \bigl (d(\chi \cdot \pi^* \theta) \bigr) 
= \lim _{t \to +\infty} d\bigl((\chi \circ \varphi_t) \cdot \pi^* \theta \bigr)
= d\pi^* \theta = \pi ^* d\theta,$$
in the sense of currents.
\end{proof}
We also obtain the following norm estimates. Recall from Proposition 2.4 that every $\psi \in \mathcal Z^{p,k}(R, \xi)$ can be written as $\psi = \pi^* T$ for some (unique) $T \in \mathcal D'^k(\mathbf R^n)$.
\begin{proposition}\label{real-prop} Let $k \in \{1, \dots , n\}$, $\frac{n}{k}<p <\frac{n}{k-1}$, and $(q, \ell)$ be the Poincar\'e dual of $(p, k)$. There exists some constant $C >0$ such that
the norm of every current $\pi ^*(T) \in \mathcal Z^{p,k}(R, \xi)$ satisfies 
$$\Vert \pi ^* T \Vert_{\Psi^{p, k}(R)} \asymp _C \sup \Bigl\{T(\theta) : \theta \in \Omega _c ^{\ell-1}(\mathbf R^n),~ \Vert \pi ^* d\theta \Vert_{\Psi^{q, \ell}(R)} \leqslant 1 \Bigr\}.$$
Moreover for $\theta \in \Omega _c ^{k-1}(\mathbf R ^n)$, the norm of the form $\pi^*d\theta 
\in \mathcal Z^{p,k}(R, \xi)$ satisfies 
$$\Vert \pi ^* d\theta \Vert_{\Psi^{p, k}(R)} 
 \leqslant C \inf _{t \in \mathbf R} \bigl\{ e^{-(k -\frac{n}{p})t} \Vert d\theta \Vert _{L^p \Omega ^{k}(\mathbf R^n)}
 + e^{(1-k +\frac{n}{p})t} \Vert \theta \Vert _{L^p \Omega ^{k-1}(\mathbf R^n)} \bigr\}.$$
 \end{proposition}
Observe that the exponents in the last inequality satisfy: $k -\frac{n}{p}>0$ and $1-k +\frac{n}{p} >0$.
\begin{proof} The spaces $\mathcal Z^{p,k}(R, \xi)$ and $L^p \mathrm{H}^{k} _{\mathrm{dR}}(R)$ are Banach isomorphic by the map $[\omega] \mapsto \omega_\infty$ of Prop. \ref{dR-identification} (2). 
Thus 
the form $\pi^*d\theta 
\in \mathcal Z^{p,k}(R, \xi)$ and the class $[d(\chi \cdot \pi ^* \theta)] \in L^p \mathrm{H}^{k} _{\mathrm{dR}}(R)$
have comparable norms.
The inequalities follow then from Theorem \ref{Lie-thm}(4), Proposition \ref{Lie-inequalities} and Lemma \ref{Lie-lemma}, applied with $H = \mathbf R^n$ and $\delta = -t \mathrm{id}_{\mathbf R^n}$. 
\end{proof}

\section{The groups $S _\alpha \in \mathcal S ^{2,3}_{\mathrm{straight}}$}  
\label{s - non QI solvable}

We prove Theorem \ref{S_mu-thm} (stated in the introduction) that determines the second $L^p$-cohomology of the groups $S _\alpha \in \mathcal S ^{2,3}_{\mathrm{straight}}$. 

\subsection{Reduction and decomposition}
\label{ss - construction} 

Let $S _\alpha = \mathbf R ^2 \ltimes _\alpha \mathbf R^3 \in \mathcal S ^{2,3}_{\mathrm{straight}}$. We provide a somewhat ``canonical'' presentation of the group $S _\alpha$ (see Proposition \ref{S_mu-prop}), which is then used to decompose $S _\alpha$.

Recall that precomposing $\alpha$ with an element of $\mathrm{GL}_2 (\mathbf R)$, or postcomposing with an element of the permutation group $S_3$, does not affect the isomorphism class of $S _\alpha$. 

Denote by $\{\varepsilon _1, \varepsilon _2\}$ the canonical basis of $\mathbf R ^2$ and by $\{\varepsilon ^* _1, \varepsilon ^*_2\}$ its dual basis. By definition of the family $\mathcal S ^{2,3}_{\mathrm{straight}}$, the weights $\varpi _{1}, \varpi _2, \varpi _3$ of $\alpha$ generate $(\mathbf R ^2)^*$ and belong to a line disjoint from $0$. By precomposing $\alpha$ by an element of $\mathrm{GL}_2 (\mathbf R)$ and postcomposing by a permutation of the diagonal entries, if necessary, we can ensure that:
\begin{itemize}
\item the $\varpi _i$'s belong to the vertical line $\Delta$ passing through $-\varepsilon _1 ^*$,
\item they admit  $-\varepsilon _1 ^*$ as center of mass,
\item the algebraic distances between them on $\Delta$ (oriented by $\varepsilon _2 ^*$) satisfy: $0 \leqslant \varpi _2-\varpi _3 \leqslant \varpi _1-\varpi _2$; {\it i.e.\!} $\varpi _{3}, \varpi _2, \varpi _1$ lie in this order on $\Delta$, and $\varpi _2$ is closer to $\varpi _3$ than to $\varpi _1$. 
\end{itemize}
In other words, there exist $\mu _1, \mu _2, \mu _3 \in \mathbf R$, not all equal, so that the weights can be written  
$\varpi _i = -\varepsilon ^*_1 + \mu _i \varepsilon ^*_2$, for $i=1,2,3$, with:
\begin{equation}\label{mu_eqn}
\sum _{i=1} ^3 \mu _i =0~~\mathrm{~~and~~}~~~ 0 \leqslant \mu _2-\mu _3  \leqslant\mu _1-\mu _2.
\end{equation}

To sum up the above reduction, we have established:
\begin{proposition} \label{S_mu-prop} There exists $D _\mu = \mathrm{diag}(\mu _1, \mu _2, \mu _3) \in \mathrm{Diag}(\mathbf R ^3)$, a non-zero diagonal matrix, unique up to a positive multiplicative constant, enjoying the relations (\ref{mu_eqn}),
such that, up to precomposition by an element of $\mathrm{GL}_2 (\mathbf R)$ and postcomposition by a permutation of the diagonal entries, 
the Lie group morphism $\alpha$ can be written:
$$\alpha(t,s) = e^{-t I_3 + s D_\mu}, ~\mathrm{for~~every}~~ (t,s) \in \mathbf R ^2.$$
Moreover, with the notation of Theorem \ref{S_mu-thm}, one has: 
$$p_\alpha = 1 + \frac{\mu _1-\mu _3}{\mu _1-\mu _2}.$$
\end{proposition}

We assume from now on that the expression of $\alpha$ is as in the previous proposition.
Let consider the following subgroups of $S _\alpha$:
$$R:=\mathbf R \ltimes _{-I_3} \mathbf R^3 \quad \hbox{\rm and} \quad H_\mu := \mathbf R\ltimes _{D_\mu} \mathbf R^3.$$ 
The group $R$ is naturally isometric to $\mathbb H ^4_{\mathbf R}$. 
Let $\frak r$ and $\frak h_\mu $ be their Lie algebras.
Let $(0, D_\mu)$ and $(0, -I_3)$ denote the derivations of $\frak r$ and $\frak h_\mu $ that trivially extend $-I_3$ and $D_\mu$. 
Then $S_\alpha$ admits two decompositions, namely:
$$S_\alpha~=~ \mathbf R \ltimes_{(0, D_\mu)}  R \quad \hbox{\rm and} \quad S_\alpha ~=~ \mathbf R \ltimes _ {(0, -I_3)} H_\mu.$$
We denote again by $\xi$ the left-invariant vector field on $R$ carried by the $\mathbf R$-factor, and by $\pi$ the projection map from $R$ onto $\mathbf R ^3$. 

The proof of Theorem \ref{S_mu-thm} will mainly rely on the decomposition $S_\alpha = \mathbf R \ltimes_{(0, D_\mu)}  R$, in combination with the description of the $L^p$-cohomology of $R \simeq \mathbb H ^4_{\mathbf R}$ given in Theorem \ref{real-thm} and Proposition \ref{real-prop}. We will use the realization of de Rham $L^p$-cohomology by means of currents, which gives the Banach space isomorphism:
$$L^p {\rm H}^2_{\rm dR}(R)\simeq \mathcal Z^{p,2}(R, \xi),$$
where $\mathcal Z^{p,2}(R, \xi)$ is the space of closed $2$-currents $\psi$ on $R$, invariant under the flow $(\varphi_t)$ of $\xi$ and such that $\Vert \psi \Vert_{\Psi^{p,2}} < + \infty$.
According to Proposition \ref{boundary-prop}, every $\psi \in \mathcal Z ^{p,2} (M, \xi)$ can be written $\psi = \pi ^* (T)$ for some $T \in \mathcal D '^{2}(\mathbf R^3) \cap \Ker d$.

\subsection{First cohomological observations} \label{S_mu-obs} 
We derive from previous results some preliminary observations on $L^p \mathrm{H}_{\mathrm{dR}}^2 (S_\alpha)$ whose statements do not depend on the morphism $\alpha$.
The notations are the same as in the previous section.

\begin{proposition}\label{S_mu-obs1} One has $L^p \mathrm{H}_{\mathrm{dR}}^2 (S_\alpha) = \{0\}$ for $p < \frac{3}{2}$.
\end{proposition}

\begin{proof} One has $S_\alpha = \mathbf R \ltimes _\delta H_\mu$, with $\delta = (0, -I_3) \in \mathrm{Der}(\frak h_\mu)$. The ordered list of eigenvalues of $-\delta$ enumerated with multiplicity, is
$$\lambda _1 =0< \lambda _2 = \lambda _3 =\lambda _4 =1.$$
Thus, with the notations of Section \ref{s - Lie}, the trace of $-\delta$ is $h=3$, and one has $W _{2} = \lambda _3 + \lambda _4=2$. Therefore the statement follows from Theorem \ref{Lie-thm}(1).
\end{proof}

 \begin{proposition}\label{S_mu-obs2} For $p \in (\frac{3}{2}; 3)$, the Banach space $\mathcal Z^{p,2}(R, \xi)$ is non-zero, and there exists a linear isomorphism
$$L^p \mathrm{H}_{\mathrm{dR}}^2 (S_\alpha ) \simeq \Bigl\{\pi ^* T \in \mathcal Z^{p,2}(R, \xi) : \int _{\mathbf R} \Vert \pi ^* {e^{sD_\mu}}^* T \Vert^p _{\Psi ^{p,2}(R)} ds < +\infty \Bigr\}.$$
\end{proposition}

\begin{proof} When $p \in (\frac{3}{2}; 3)$, Theorem \ref{real-thm} shows that $L^p \mathrm{H}_{\mathrm{dR}}^2 (R)$ is non-zero and Hausdorff, and that $L^p \mathrm{H}_{\mathrm{dR}}^k (R) = \{0\}$ in all degrees $k\neq 2$. Since $S_\alpha = \mathbf R \ltimes _{(0, D_\mu)} R$, the above description of the cohomology of $R$, in combination with a Hochschild-Serre spectral sequence argument (see \cite[Corollary 6.10]{BR2}), yields the following linear isomorphism
$$L^p \mathrm{H}_{\mathrm{dR}}^2 (S_\alpha) \simeq \Bigl\{[\omega ] \in L^p \mathrm{H}_{\mathrm{dR}}^2 (R) : \int _{\mathbf R} \Vert {e^{s(0, D_\mu)}}^* [\omega] \Vert^p _{L^p \mathrm{H}^ 2(R)} ds < +\infty \Bigr\}.$$
By Theorem \ref{real-thm}(2), the Banach spaces $L^p \mathrm{H}_{\mathrm{dR}}^2 (R)$ and $\mathcal Z^{p,2}(R, \xi)$ are isomorphic. 
Moreover every $\psi \in \mathcal Z ^{p,2} (R, \xi)$ can be written $\psi = \pi ^* (T)$ for some $T \in \mathcal D '^{2}(\mathbf R^3) \cap \Ker d$ (see Proposition \ref{boundary-prop}). This leads to the desired linear isomorphism.
\end{proof}

\begin{proposition}\label{S_mu-obs3} For $p>3$, the space $L^p \overline{\mathrm{H}_{\mathrm{dR}}^2} (S_\alpha)$ is non-zero. 
\end{proposition}

\begin{proof} Consider again $\lambda _1 =0< \lambda _2 = \lambda _3 = \lambda _4 =1$ the list of the eigenvalues of $-\delta = - (0, -I_3) \in \mathrm{Der}(\frak h_\mu)$. The trace of $-\delta$ is $h=3$, and one has $w _{2} = \lambda _{1} + \lambda _{2} = 1$. 
Since the rank of $S_\alpha$ is equal to $2$, it follows from \cite[Theorem B and Corollary 3.4]{BR2} that $L^p \overline{\mathrm{H}_{\mathrm{dR}}^2} (S_\alpha )$ is non-zero for $p>\frac{h}{w _{2}}=3$.
\end{proof}

\subsection{Non-vanishing of the second $L^p$-cohomology}
\label{ss - non-vanishing}
We wish to establish the non-vanishing part of Theorem \ref{S_mu-thm}. 
Thanks to Proposition \ref{S_mu-obs3}, we just need to prove that $L^p \mathrm{H}_{\mathrm{dR}}^2 (S_\alpha) \neq \{0\}$ for $p \in (p_\alpha; 3)$. 

Assume $p \in (\frac{3}{2}; 3)$. By Proposition \ref{S_mu-obs2} there is a linear isomorphism 
$$L^p \mathrm{H}_{\mathrm{dR}}^2 (S_\alpha) \simeq \Bigl\{\pi ^* T \in \mathcal Z^{p,2}(R, \xi) : \int _{\mathbf R} \Vert {\pi ^* e^{sD_\mu }}^* T \Vert^p _{\Psi ^{p,2}(R)} ds < +\infty \Bigr\}.$$ Recall from Theorem \ref{real-thm}(2) that the space $\mathcal Z^{p,2}(R, \xi)$ contains the forms $\pi ^* d\theta$, with $\theta \in \Omega^1_c({\mathbf R}^3)$.
Therefore in order to show that $L^p \mathrm{H}_{\mathrm{dR}}^2 (S_\alpha) $ is non-zero, it is enough to exhibit a non-zero form $d \theta$ with $\theta \in \Omega ^1 _c (\mathbf R ^3)$ and $\Vert {\pi ^* e^{sD_\mu }}^* d\theta \Vert _{\Psi ^{p,2}(R)} \to 0$ exponentially fast when $s \to \pm \infty$.

Let $\theta$ be a smooth compactly supported $1$-form on $\mathbf R ^3$. From Proposition \ref{real-prop}, with $\theta$ replaced by ${e ^{s D_\mu}}^* \theta$, we have: 
\begin{align*}
\Vert \pi ^* {e ^{s D_\mu}}^* d\theta  \Vert_{\Psi ^{p,2}(R)} 
\leqslant 
C \inf _{t \in \mathbf R} \bigl\{e^{-(2-{3\over p})t}& \Vert  {e ^{s D_\mu}}^* d\theta \Vert _{L^p \Omega ^2({\bf R}^3)} \\+ &~~e^{(-1+{3 \over p})t} \Vert  {e ^{s D_\mu}}^*\theta \Vert _{L^p \Omega ^{1}({\bf R}^3)}\bigr\}.
\end{align*}
Write $\theta = f dx + g dy + h dz$, so that $d\theta = F dy \wedge dz + G dx \wedge dz + H dx \wedge dy$. 
Since the trace of $D_\mu$ is zero, Lemma \ref{Lie-lemma} applied with $H = \mathbf R ^3$ and $\delta = D_\mu = \mathrm{diag}(\mu _1, \mu _2, \mu _3)$ gives the following estimates: 
$$ \Vert {e ^{s D_\mu}}^* d\theta \Vert_{L^p \Omega ^2} \asymp e^{-s\mu _1} \Vert F \Vert_{L^p} + e^{-s\mu_2} \Vert G \Vert_{L^p} + e^{-s\mu_3} \Vert H \Vert_{L^p},$$
$$ \Vert {e ^{s D_\mu}}^*\theta \Vert _{L^p \Omega ^1} \asymp e^{s\mu_1} \Vert f \Vert_{L^p} + e^{s\mu_2} \Vert g \Vert_{L^p} + e^{s\mu_3} \Vert h \Vert_{L^p}.$$
We denote by $\alpha_\pm$ the exponent of the leading term in the asymptotics of $\Vert {e^{s D_\mu}}^* d\theta \Vert $ when $s \to \pm\infty$, namely 
$\Vert {e^{s D_\mu}}^* d\theta \Vert \asymp_{s \to \pm\infty} e^{\alpha_\pm s}$; note that $\alpha_+$ and $\alpha_-$ are opposites of diagonal coefficients of $D_\mu$ since ${\rm trace}(D_\mu)=0$. 
Similarly, we denote by $\beta_\pm$ the exponent of the leading term in the asymptotics of $\Vert {e^{s D_\mu}}^* \theta \Vert $ when $s \to \pm\infty$, namely 
$\Vert {e^{s D_\mu}}^* \theta \Vert \asymp_{s \to \pm\infty} e^{\beta_\pm s}$; note that $\beta_+$ and $\beta_-$ are diagonal coefficients of the matrix $D_\mu$. One has:

\begin{lemma} 
\label{lemma - num}
Let $a,b>0$ be positive real numbers and let $\alpha, \beta \in {\mathbf R}$. 
We assume that $A=A(s) \asymp e^{\alpha s}$ and $B=B(s) \asymp e^{\beta s}$ when $s \to +\infty$ (resp. when $s \to -\infty$). 
Then $\inf_{t \in {\mathbf R}} \{ e^{-at} A + e^{bt} B \} \asymp e^{{a\beta + b\alpha \over a+b}s}$, when $s \to +\infty$ (resp. when $s \to -\infty$). 
In particular, the infimum tends to $0$ if and only if we have $(a\beta + b\alpha) s \to -\infty$, when $s \to +\infty$  (resp. when $s \to -\infty$); in which case the speed of convergence to $0$ is exponential. 
\end{lemma}
\begin{proof} 
Assume first that $A,B>0$ are fixed and consider the function $f$ defined by $f(t) = e^{-{at}} A + e^{bt} B$. 
We have $\lim_{t \to \pm\infty} f(t) = +\infty$ so $f$ achieves its minimum at a point $t_{\rm min}$ such that $f'(t_{\rm min}) = 0$. 
Since $f'(t) = -ae^{-{at}} A + be^{bt} B$, we have ${A \over B} = {b \over a} e^{(a+b)t_{\rm min}}$ and therefore the minimal value of $f$ is: 
$$f(t_{\rm min})=Be^{bt_{\rm min}} ({A \over B}e^{-(a+b)t_{\rm min}} + 1 ) = B e^{bt_{\rm min}}({b \over a} + 1).$$
When $A \asymp e^{\alpha s}$ and $B \asymp e^{\beta s}$, we have $e^{(\alpha - \beta)s} \asymp {A \over B} = {b \over a} e^{(a+b)t_{\rm min}}$. 
Then $f(t_{\rm min}) \asymp e^{\beta s} e^{b t_{\rm min}} \asymp e^{(\beta + b {\alpha - \beta \over a+b})s}= e^{{a\beta + b\alpha \over a+b}s}$. 
\end{proof}
For $p \in (\frac{3}{2}; 3)$, with $a= 2-{3 \over p}$ and $b = -1 + {3 \over p}$ in the above lemma, we obtain the following estimate when $s \to \pm\infty$: 
$$\Vert {\pi ^* e^{sD_\mu }}^* d\theta \Vert_{\Psi ^{p,2}(R)} \lesssim e^{{a\beta_\pm + b\alpha_\pm \over a+b}s}=e^{\{(2p-3)\beta_\pm + (-3+p)\alpha_\pm \}\frac{s}{p}}.$$
This shows that for the condition $\int _{\mathbf R} \Vert {\pi ^* e^{sD_\mu }}^* T \Vert^p _{\Psi ^{p,2}(R)}
ds < +\infty$ to be satisfied by $T = d\theta$, it is sufficient to have: 
\begin{equation}\label{S_mu-ineq}
(2p-3)\beta_+ + (3-p)\alpha_+  < 0 ~~\mathrm{~~~and~~~}~~ (2p-3)\beta_- + (3-p)\alpha_- > 0.
\end{equation}
To exhibit such a form $d \theta$, we will use the 

\begin{lemma}
There exist forms $\theta = f dx$ and $\Theta = g dy + h dz$ in $\Omega^1_c({\bf R}^3)$, such that $d\theta = d\Theta = G dx \wedge dz + H dx \wedge dy \neq 0$.
\end{lemma}

\begin{proof}
Let $u$ be an arbitrary non-zero function in $\Omega _c ^0 (\mathbf R ^3)$. Its differential is $du = {\partial u\over \partial x} dx+ {\partial u\over \partial y}dy + {\partial u \over \partial z}dz$. Set $\theta:= {\partial u\over \partial x}dx$ and $\Theta := -{\partial u\over \partial y}dy -{\partial u \over \partial z}dz$. Since $ddu =0$,
one has 
$d\theta = d\Theta = -{\partial^2 u \over {\partial x \partial z}} dx \wedge dz -{\partial^2 u \over {\partial x \partial y}}dx \wedge dy.$
\end{proof}

Let $d\theta = d\Theta$ be as in the previous lemma. From the relations (\ref{mu_eqn}) we have $\alpha_+ = -\mu_3$ and $\alpha_- = -\mu_2$. 
Similarly we have $\beta_- = \beta_-(\theta) = \mu_1$ and $\beta_+ = \beta_+(\Theta) = \mu_2$. 
When $s \to +\infty$, the integrability conditions (\ref{S_mu-ineq}) and the relations (\ref{mu_eqn}), lead to the following condition 
$$ 
(2p-3)\mu_2 - (3-p)\mu_3 < 0,$$ 
hence
$p(2\mu_2 + \mu_3) -3\mu_2 -3\mu_3< 0$, that is $p(\mu_2 -\mu_1) + 3\mu_1 >0$, 
amounting to 
$$p > {3\mu_1 \over \mu_1 -\mu _2}  = 1 + {\mu_1 -\mu_3 \over \mu_1 -\mu_2}= p_\alpha.$$
When $s \to -\infty$, they lead to 
$$(2p-3)\mu_1 - (3-p)\mu_2 > 0,$$
hence 
$p(2\mu_1 + \mu_2) - 3 \mu_1 -3\mu_2  > 0$, that is $p(\mu_1-\mu_3) + 3 \mu_3 >0$,
amounting to 
$$p > {-3\mu_3 \over \mu_1 -\mu _3}  =1+{\mu_2 -\mu_3 \over \mu_1 -\mu_3}.$$
The latter condition is implied by the former one, since $1 +{\mu_2 -\mu_3 \over \mu_1 -\mu_3} \in [1; \frac{3}{2}]$ and $p _\alpha \in [2; 3]$. 
To sum up, we have shown that $L^p \mathrm{H}_{\mathrm{dR}}^2 (S_\alpha) \neq \{0\}$ for $p \in (p_\alpha; 3)$, as expected.

\subsection{Vanishing of the second $L^p$-cohomology}
\label{ss - Vanishing degree 2}
It remains to prove the vanishing statement in Theorem \ref{S_mu-thm}. 
It will be obtained by using a Poincar\'e duality argument, together with some estimates similar to those from the non-vanishing part. 

According to Proposition \ref{S_mu-obs1}, it is enough to consider the case $p \in (\frac{3}{2}; 3)$. We start again from the identification given in Proposition \ref{S_mu-obs2}:
$$L^p \mathrm{H}_{\mathrm{dR}}^2 (S_\alpha) \simeq \Bigl\{\pi ^* T \in \mathcal Z^{p,2}(R, \xi) : \int _{\mathbf R} \Vert {\pi ^* e^{sD_\mu }}^* T \Vert^p _{\Psi ^{p,2}(R)} ds < +\infty \Bigr\}.$$ 
To show that $L^p \mathrm{H}_{\mathrm{dR}}^2 (S_\alpha)$ vanishes, it is enough to prove that every $\pi ^* T \in \mathcal Z^{p,2}(R, \xi)$ satisfies $\Vert {\pi ^* e^{sD_\mu }}^* T \Vert _{\Psi ^{p,2}(R)} \to +\infty$, when $s \to +\infty$ or when $s \to -\infty$. 
Recall from Proposition \ref{real-prop}, that: 
$$\Vert \pi ^* T \Vert _{\Psi ^{p,2}(R)} \asymp \sup \{ T(\theta) : \theta \in \Omega^1_c({\bf R}^3) , ~\Vert \pi^*d\theta) \Vert_{\Psi ^{q,2}(R)} \leqslant 1 \},$$
where $q$ denotes the H\"older conjugate of $p$.
When replacing $\pi ^* T$ by $\pi ^* {e^{sD_\mu }}^* T$ with $s \in {\bf R}$, a change of variable provides: 
$$\Vert \pi ^* {e^{sD_\mu }}^* T \Vert _{\Psi ^{p,2}(R)} \asymp \sup \{ T(\theta) : \theta \in \Omega^1_c({\bf R}^3)   ,~ \Vert \pi^*{e^{sD_\mu }}^*d\theta) \Vert_{\Psi ^{q,2}(R)} \leqslant 1 \}.$$
By Proposition \ref{real-prop}, with $\theta$ replaced by ${e^{s D_\mu}}^* \theta$, we obtain: 
\begin{align*}
\Vert \pi ^* {e ^{s D_\mu}}^* d\theta  \Vert_{\Psi ^{q,2}(R)} 
\leqslant 
C \inf _{t \in \mathbf R} \bigl\{e^{-({3 \over p}-1)t}& \Vert  {e ^{s D_\mu}}^* d\theta \Vert _{L^q \Omega ^2({\bf R}^3)} \\+ &~~e^{(2-{3\over p})t} \Vert  {e ^{s D_\mu}}^*\theta \Vert _{L^q \Omega ^{1}({\bf R}^3)}\bigr\},
\end{align*}
where, in the upper bound, the coefficients in the exponentials in front of the norms come from the identity $2 - {3 \over q} = {3 \over p}-1$ and $1 + {3 \over q} = 2 - {3 \over p}$. 

Using Lemma \ref{lemma - num} with $a={3 \over p}-1$ and $b = 2-{3\over p}$, and keeping the notation $\alpha_\pm$ and $\beta_\pm$ defined after replacing $L^p$ norms by $L^q$ norms, we obtain 
\begin{equation}\label{S_mu-equation}
\Vert \pi ^* {e ^{s D_\mu}}^* d\theta  \Vert_{\Psi ^{q,2}(R)}\lesssim e^{(a\beta_\pm + b\alpha_\pm)s},
\end{equation}
when $s \to \pm \infty$. These observations lead to the
\begin{lemma}
Let $\pi ^* T \in \mathcal Z^{p,2}(R, \xi)$ and let $T = T_1 dy \wedge dz + T_2 dx \wedge dz + T_3 dx \wedge dy$ be its writing in the canonical global coordinates of ${\bf R}^3$, with $T_i \in \mathcal D '^0 (\mathbf R^3)$. 
\begin{itemize}
\item[{\rm (i)}] If $T_1 \neq 0$, then $\displaystyle \lim_{s \to -\infty}\Vert \pi ^* {e^{sD_\mu }}^* T \Vert _{\Psi ^{p,2}(R)} = +\infty$. 
\item[{\rm (ii)}] If $T_3 \neq 0$ and if $p < p_\alpha$, then $\displaystyle \lim_{s \to +\infty}\Vert \pi ^* {e^{sD_\mu }}^* T \Vert _{\Psi ^{p,2}(R)} = +\infty$. 
\end{itemize} 
\end{lemma}

\begin{proof}
(i). Let $\theta$ be of the form $\theta = f dx$ with $f$ a smooth compactly supported function such that $T_1(f)=1$. 
Then $T(\theta) = T_1(f) = 1$. 
We have $d\theta = -{\partial f \over \partial y} dx \wedge dy - {\partial f \over \partial z} dx \wedge dz$. 
Therefore with the relations (\ref{mu_eqn}), one has $\alpha_- = -\mu_2$ and $\beta_- = \mu_1$. 
We consider the quantity 
$$p(a \beta_- + b \alpha_-) = (3-p)\mu_1 - (2p-3)\mu_2 = -p(\mu_2 - \mu_3) -3\mu_3.$$
Since $p \in ({3 \over 2};3)$ and since $\mu_3 \leqslant \mu_2 \leqslant 0$, one has  $-p(\mu_2 - \mu_3) -3\mu_3 \geqslant -3\mu_2 \geqslant 0$. We deduce that $a \beta_- + b \alpha_- >0$. 
Relation (\ref{S_mu-equation}) then implies that $\displaystyle \lim_{s \to -\infty}\Vert \pi ^* {e ^{s D_\mu}}^* d\theta  \Vert_{\Psi ^{q,2}(R)}= 0$, hence that 
$\displaystyle \lim_{s \to -\infty} \Vert \pi ^* {e^{sD_\mu }}^* T \Vert _{\Psi ^{p,2}(R)}= +\infty$. 

(ii). Let $\theta$ be of the form $\theta = h dz$ with $h$ a smooth compactly supported function such that $T_3(h)= 1$. 
Then $T(\theta) = T_3(h) = 1$. 
We have $d\theta = {\partial h \over \partial x} dx \wedge dz + {\partial h \over \partial y} dy \wedge dz$. 
Thus $\alpha_+ = -\mu_2$ and $\beta_+ = \mu_3$. 
We consider the quantity 
$$p(a \beta_+ + b \alpha_+) = (3-p)\mu_3 - (2p-3)\mu_2= p(\mu_1 - \mu_2) -3\mu_1.$$ 
One has $a \beta_+ + b \alpha_+ <0$ if and only if 
$$p < {3\mu_1 \over \mu_1 -\mu _2}  = 1 + {\mu_1 -\mu_3 \over \mu_1 -\mu_2}= p_\alpha.$$ 
Thus for $p < p_\alpha$, we have $a \beta_+ + b \alpha_+ < 0$. 
Relation (\ref{S_mu-equation}) then implies that $\displaystyle \lim_{s \to +\infty}\Vert \pi ^* {e ^{s D_\mu}}^* d\theta  \Vert_{\Psi ^{q,2}(R)}= 0$, hence that 
$\displaystyle \lim_{s \to +\infty} \Vert \pi ^* {e^{sD_\mu }}^* T \Vert _{\Psi ^{p,2}(R)}= +\infty$. 
\end{proof} 

We can now turn to the 
\begin{proof}[Proof of $L^p \mathrm{H}_{\mathrm{dR}}^2 (S_\alpha)=\{0\}$ for $p \in (\frac{3}{2}; p_\alpha)$]
Let $\pi ^* T \in \mathcal Z^{p,2}(R, \xi)$ be in correspondence with a non-zero class in $L^p {\rm H}^2_{\rm dR}(S_\alpha)$. 
By Item (i) in the previous lemma, we must have $T_1 = 0$, which implies that both $T_2$ and $T_3$ are $\neq 0$ because $T$ is closed. 
By Item (ii) in the previous lemma, we must have $p \geqslant  p_\alpha$. 
This proves the vanishing by contraposition. 
\end{proof}

\section{Complex hyperbolic spaces}  
\label{s - complex hyp sp}
Another family of concrete examples for which a strip decomposition as stated in Theorem \ref{Lie-thm} can be derived, is provided by the so-called complex hyperbolic spaces. Theorem \ref{complex-thm} below describes the regions where the cohomology vanishes or not. It also states that the cohomology is Hausdorff. The proof of the (non-)vanishing statement relies on Theorem \ref{Lie-thm}, in combination with some additional analysis on Heisenberg groups that is developed in Section \ref{complex-diff}. The Hausdorff statement is a deep result due to Pansu \cite[Th\'eor\`eme 1]{Pa09}. We refer to his paper for a proof. The section ends with some complementary results (Propositions \ref{null-prop} and \ref{dense-prop}) that provide a finer description of the cohomology.

Let $\mathrm{Heis}(2m-1)$ be the Heisenberg group of dimension $2m-1$ ($m\geqslant 2$), {\it i.e.}\! the simply connected nilpotent Lie group whose Lie algebra $\frak n$ admits 
$$X_1, \dots, X_{m-1}, Y_1, \dots, Y_{m-1}, Z$$ 
as a basis, and where the only non-trivial relations between the above generators are $[X_i, Y_i] = Z$, for all $i \in \{1, \dots, m-1\}$.
Let $R = {\mathbf R} \ltimes_\delta \!N$ with $N=\mathrm{Heis}(2m-1)$ and $\delta = - \mathrm{diag}(1, \dots, 1, 2) \in \mathrm{Der}(\frak n)$. Then $R$ is a solvable Lie group isometric to the complex hyperbolic space $\mathbb H^{m} _{\bf C}$.  

Apart from the density statement and the (non-)vanishing statement in Items (2) and (3) -- when $m \geqslant 3$ --, the following result already appears in \cite{Pa99,Pa09}.

\begin{theorem}\label{complex-thm} 
Let $k\in \{1, \dots , 2m-1\}$. One has: 
\begin{enumerate}
 \item $L^p \mathrm{H}^k _{\mathrm{dR}}(R) =\{0\}$ for $1< p<\frac{2m}{k+1}$ or $p>\frac{2m}{k-1}$.
 \item If $\frac{2m}{k+1} < p< \frac{2m}{k}$, then 
 $L^p \mathrm{H}^{k} _{\mathrm{dR}}(R)$ is Hausdorff and Banach isomorphic to $\mathcal Z^{p,k}(R, \xi)$. Moreover $L^p \mathrm{H}^{k} _{\mathrm{dR}}(R) \neq \{0\}$ if and only if $k\geqslant m$.
 \item If $\frac{2m}{k} < p< \frac{2m}{k-1}$, then $L^p \mathrm{H}^{k} _{\mathrm{dR}}(R)$ is Hausdorff and
the classes of the $d(\chi \cdot \pi ^* \theta)$'s 
(where $\theta \in \Omega _c ^{k-1}(N)$) form a dense subspace. Moreover $L^p \mathrm H _{\mathrm{dR}} ^{k} (R)
\neq \{0\}$ if and only if $k\leqslant m$.
\end{enumerate}
\end{theorem}

We notice that the above statement will be complemented in Section \ref{complex-complement}: Proposition \ref{null-prop} will describe the zero-elements among the classes $[d(\chi \cdot \pi ^* \theta)]$'s, while Proposition \ref{dense-prop} will exhibit a natural dense subspace in  $\mathcal Z^{p,k}(R, \xi)$.

\begin{proof}[Beginning of proof of Theorem \ref{complex-thm}]
For $k\in \{0, \dots , 2m-2\}$, one has $h =2m$, $w_k = k$, and $w_{2m-1}= h \geqslant 2m-1$. Similarly $W_k 
= k+1$ for $k\in \{1, \dots , 2m-1\}$ and $W_0 =0 \le 1$.
Item (1) comes from Theorem \ref{Lie-thm}(1). 
The first part of Item (2) follows from Theorem \ref{Lie-thm}(2). 
The Hausdorff statement in Item (3) is a deep theorem of Pansu \cite[Th\'eor\`eme 1]{Pa09}, in combination with Poincar\'e duality (Proposition \ref{dR-Pd}).
The density property follows from Theorem \ref{Lie-thm}(3).

It remains to establish the (non-)vanishing parts of Items (2) and (3). It will require more material, and the proof will be completed only at the end of the section. 
\end{proof}

\subsection{Differential forms on the Heisenberg group} \label{complex-diff}
We complete the proof of the (non-)vanishing statements in Theorem \ref{complex-thm}. 
They rely on two lemmata (Lemma \ref{complex-lem1} and \ref{complex-lem2} below). 
The material is inspired by Rumin's paper \cite{Rumin}.

Recall that $\frak{n}$ denotes the Lie algebra of $N= \mathrm{Heis}(2m-1)$. It decomposes as $\frak{n} = \frak{n}_1 \oplus \frak{n}_2$, where $\frak{n}_1 :=\mathrm{Span}(X_1, \dots, X_{m-1}, Y_1, \dots, Y_{m-1})$ and $\frak{n}_2 :=\mathrm{Span}(Z)$ are respectively the eigenspaces of $-\delta$ corresponding to the eigenvalues $1$ and $2$.

Let $x_1, \dots, x_{m-1}, y_1, \dots, y_{m-1}, z$ be the coordinates on $N$ induced by the exponential map $\frak{n} \to N$. Let $\tau := dz -\frac{1}{2}\sum _{i=1}^{m-1} (x_i dy_i - y_i dx_i)$.
We identify  $\frak{n}^*$ with the space of the left-invariant $1$-forms on $N$. One has $\frak{n}^* = \frak{n}_1^* \oplus \frak{n}_2 ^*$, where $\frak{n}_1^* :=\mathrm{Span}(dx_1, \dots, dx_{m-1}, dy_1, \dots, dy_{m-1})$ and $\frak{n}_2^* :=\mathrm{Span}(\tau)$ are the eigenspaces of $-\delta$ of eigenvalues $1$ and $2$. 

The form $d\tau = -\sum _{i=1}^{m-1} dx_i \wedge dy_i$ is a symplectic form when restricted to $\frak{n}_1$. 
Therefore the Lefschetz map 
\begin{equation}\label{Lefschetz}
L_k : \wedge ^k \frak{n}_1^* \to \wedge ^{k+2} \frak{n}_1^*,
~~\alpha \mapsto \alpha \wedge d\tau,
\end{equation}
is injective for $k \leqslant m-2$ and surjective for $k\geqslant m-2$,
see \cite[Proposition 1.1]{BBG}.

The weight decomposition  $\wedge ^k \frak{n}^* = \wedge ^k \frak{n}_1^* \oplus \wedge ^{k-1} \frak{n}_1^* \wedge \tau$
associated to $-\delta$, yields a decomposition
$$\Omega ^k  (N) = \Omega ^k _{1} \oplus \Omega ^k _{2}
,$$
with $\Omega ^k _{2} = \Omega ^{k-1} _{1} \wedge \tau$. 
Therefore, every $\theta \in \Omega ^k(N)$ decomposes uniquely as
$$\theta = \theta_1 + \theta _2\wedge \tau, ~~\mathrm{with}~~
\theta_1 \in \Omega ^k _{1}~~\mathrm{and}~~
\theta_2 \in \Omega ^{k-1} _{1}.$$
The form $\theta _1$ is said to be \emph{horizontal} and of \emph{pure weight} $k$.
The form $\theta _2 \wedge \tau$ is said to be \emph{vertical} and of \emph{pure weight} $k+1$.
We will use the indices ${}_1$ and ${}_2$ to specify the components of differential forms according to the above direct sum. 

When $\frac{2m}{k} < p< \frac{2m}{k-1}$, and according to what we have already proved for Theorem \ref{complex-thm}(3), the space $L^p \mathrm{H}^{k} _{\mathrm{dR}}(R)$ is Hausdorff and contains the dense subspace 
$\{[d(\chi \cdot \pi ^* \theta)] : \theta \in \Omega _c ^{k-1}(N)\}$. One has in addition:
 
\begin{lemma}
\label{complex-lem1}
Suppose $k\in \{2, \dots ,  2m-1\}$ and $\frac{2m}{k} < p< \frac{2m}{k-1}$.
\begin{enumerate}
\item Let $\theta \in \Omega _c ^{k-1}(N)$. If $\theta = \alpha \wedge d\tau + \beta \wedge \tau$, with $\alpha \in \Omega _c ^{k-3}(N)$ and $\beta \in \Omega _c ^{k-2}(N)$, then $[d(\chi \cdot \pi ^*\theta)] =0$ in $L^p {\mathrm H _{\mathrm{dR}} ^{k}} (R)$.
\item For $k\geqslant m+1$, one has $L^p {\mathrm H _{\mathrm{dR}} ^{k}} (R) = \{0\}$.
\end{enumerate}
\end{lemma}

\begin{proof}
(1). Suppose first that $\theta = \beta \wedge \tau$.
Since $p > \frac{2m}{k} \geqslant \frac{h}{w_k}$, Proposition \ref{Lie-inequalities} implies that 
$$\bigl\Vert [d(\chi \cdot \pi ^* \theta) ] \bigr\Vert_{L^p {\mathrm{H^k}}(R)} 
 \leqslant C \inf _{t \in \mathbf R} \bigl\{\Vert (e^{t\delta})^* \theta \Vert _{L^p \Omega ^{k-1}(N)} + \Vert (e^{t\delta})^* d\theta \Vert _{L^p \Omega ^{k}(N)}\bigr\}.$$
 Since $\theta$ is of pure weight $k$,
one has  $\Vert (e^{t\delta})^* \theta \Vert _{L^p \Omega ^{k-1}} \asymp e^{(-k + \frac{2m}{p})t}$ by Lemma \ref{Lie-lemma}. The weight of $d\theta$ is at least $k$, thus for $t \geqslant 0$ one has
$\Vert (e^{t\delta})^* d\theta \Vert _{L^p \Omega ^{k}} \lesssim 
e^{(-k + \frac{2m}{p})t}$. Therefore $\bigl\Vert [d(\chi \cdot \pi ^* \theta) ] \bigr\Vert_{L^p {\mathrm{H^k}}} \to 0$ when $t \to +\infty$,
and thus $[d(\chi \cdot \pi ^*\theta)] =0$ in $L^p {\mathrm H _{\mathrm{dR}} ^{k}} (R)$.

Now suppose that $\theta = \alpha \wedge d\tau$. 
We claim that there exists $\gamma \in \Omega _c ^{k-1} (N)$ such that $d\theta = d(\gamma \wedge \tau)$. By Corollary \ref{boundary-cor}, this will imply that 
$[d(\chi \cdot \pi ^* \theta) ] = [d(\chi \cdot \pi ^* (\gamma \wedge \tau)) ]$; which in turn implies that $[d(\chi \cdot \pi ^* \theta) ] =0$ from the previous case.
Let $\gamma :=-(-1)^{d^\circ\! \alpha}d\alpha$. One has 
$$d(\theta - \gamma \wedge \tau) =  d\alpha \wedge d\tau +(-1)^{d^\circ\! \alpha +1}(-1)^{d^\circ\! \alpha} d\alpha \wedge d\tau =0,$$
and so $d\theta = d(\gamma \wedge \tau)$.

(2). Let $\theta \in \Omega _c ^{k-1}(N)$. Since the Lefschetz map $L_i$, defined in (\ref{Lefschetz}), is surjective for $i \geqslant m-2$, one can write the weight decomposition of $\theta$ as
$$\theta = \alpha \wedge d\tau + \theta _2 \wedge \tau.$$
Therefore Item (1) implies that $[d(\chi \cdot \pi ^*\theta)] =0$ in $L^p {\mathrm H _{\mathrm{dR}} ^{k}} (R)$.
This in turn implies that $L^p {\mathrm H _{\mathrm{dR}} ^{k}} (R) = \{0\}$, thanks to the density of the $[d(\chi \cdot \pi ^*\theta)]$'s.
\end{proof}

The weight decomposition of $k$-forms can be extended to $k$-currents. This induces a decomposition $$\mathcal D'^k (N)= \mathcal D'^k_1 \oplus \mathcal D'^k_2,$$
with $\mathcal D'^k_2 = \mathcal D'^{k-1}_1 \wedge \tau$. Concretely, every $T \in \mathcal D'^k$ can be written uniquely as $$T= \sum _{\vert I\vert = \vert J\vert =k} T_{IJ} dx_I \wedge dy_J + \sum _{\vert K \vert + \vert L \vert = k-1} T_{KL} dx _K \wedge dy _L \wedge \tau,$$ with $T_{IJ}, T_{KL} \in \mathcal D'^{0}(N)$. Its weight decomposition is then $T = T_1 + T_2 \wedge \tau$, with $$ T_1 = \sum _{\vert I \vert + \vert J \vert = k} T_{IJ} dx_I \wedge dy_J ~~\mathrm{~~and~~} ~~ T_2 = \sum _{\vert K \vert + \vert L \vert = k-1} T_{KL} dx_K \wedge dy_L.$$
The current $T_1$ is said to be {\it horizontal}, and $T_2 \wedge \tau$ to be {\it vertical}. 
For $\theta \in \Omega ^{2m-1-k}_c (N)$, which weight decomposes as $\theta = \theta_1 + \theta_2 \wedge \tau$, one shows easily that
\begin{equation}\label{weight} T(\theta) = T_1(\theta _2 \wedge \tau) + T_2 (\tau \wedge \theta _1).
\end{equation}

In Theorem \ref{complex-thm}(2), we have seen that $L^p \mathrm{H}_{\mathrm{dR}}^k (R)$ is Banach isomorphic to $\mathcal Z^{p,k}(R, \xi)$,
for every $k \in \{1, \dots, 2m-1\}$ and $\frac{2m}{k+1} <p< \frac{2m}{k}$. Moreover we know from Proposition 2.4 that every $\psi \in \mathcal Z^{p,k}(R, \xi)$ can be written as $\psi = \pi^* T$ for some (unique) $T \in \mathcal D'^k(N)$. One has futhermore:

\begin{lemma} \label{complex-lem2} Let $k \in \{1, \dots, 2m-1\}$ and $\frac{2m}{k+1} <p< \frac{2m}{k}$.
\begin{enumerate}
\item For every $\pi ^*T \in \mathcal Z ^{p,k}(R, \xi)$ the $k$-current $T$ is vertical.
\item Conversely, if $\varphi \in \Omega ^{k-1} _c(N)$ is such that $d\varphi$ is vertical, then $\pi^*(d\varphi)$ belongs to $\mathcal Z^{p,k}(R, \xi)$.
\item We have $\mathcal Z^{p,k}(R, \xi) \neq \{0\}$ for every $k \in \{m, \dots, 2m-1\}$.
\end{enumerate}
\end{lemma}

\begin{proof}
(1). Every $T \in \mathcal D'^{2m-1}(N)$ is vertical, thus we can assume that $k \in \{1, \dots, 2m-2\}$.
Let $(q, \ell)$ be the Poincar\'e dual of $(p,k)$ relatively to $R$; it satisfies $\ell \in \{2, \dots, 2m-1\}$ and $\frac{2m}{\ell}<q< \frac{2m}{\ell -1}$.
For every $\pi^*T \in \mathcal Z ^{p,k}(R, \xi)$ and every vertical $\theta \in \Omega ^{\ell -1} _c (N)$, one has $T(\theta) =0$, thanks to Theorem \ref{Lie-thm} and Lemma \ref{complex-lem1}. According to Relation
(\ref{weight}), this implies that the weight decomposition of $T$ satisfies $T_1=0$. Thus $T$ is vertical.

(2). Let $\varphi \in \Omega ^{k-1} _c(N)$ be such that $d\varphi$ is vertical.
First, we claim that $d(\chi \cdot \pi ^*\varphi)$ belongs to $\Omega ^{p,k}(R) \cap \Ker d$.
One has $d(\chi \cdot \pi ^* \varphi) = d \chi \wedge \pi^* \varphi + \chi \cdot \pi ^* d \varphi$. Thus
$$\Vert d(\chi \cdot \pi ^*\varphi) \Vert _{\Omega ^{p,k}} = \Vert d(\chi \cdot \pi ^*\varphi) \Vert _{L^p \Omega ^{k}} \leqslant \Vert d\chi \wedge  \pi ^*\varphi \Vert _{L^p \Omega ^{k}}+\Vert \chi \cdot \pi ^*d\varphi \Vert _{L^p \Omega ^{k}}.$$
The form $d\chi \wedge  \pi ^*\varphi$ is compactly supported, thus it belongs to $L^p \Omega ^{k}(R)$. One has
$$\Vert \chi \cdot \pi ^*d\varphi \Vert ^p _{L^p \Omega ^{k}(R)} \leqslant
\Vert \mathbf 1 _{t \ge 0} \cdot \pi ^*d\varphi \Vert ^p _{L^p \Omega ^{k}(R)}
= \int _0 ^{+\infty} \Vert (e^{t \delta})^*d\varphi \Vert ^p _{L^p \Omega ^k (N)} ~dt.$$
Since $d \varphi$ is vertical, it is of pure weight $k+1$, and $\Vert (e^{t \delta})^*d\varphi \Vert  _{L^p \Omega ^k (N)} \asymp e ^{(-k-1 +\frac{2m}{p})t}$ by Lemma \ref{Lie-lemma}. Since $p>\frac{2m}{k+1}$, the above integral converges and the claim is proved.

One has $\pi ^* d\varphi = \lim _{t \to +\infty} \varphi _t ^*(d(\chi \cdot \pi ^*\varphi))$ in the sense of currents; thus $\pi ^*(d\varphi) \in \mathcal Z^{p,k}(R, \xi)$ by Proposition \ref{dR-identification}.

(3). It remains to show that there exists $\varphi \in \Omega_c ^k(N)$ such that $d \varphi$ is vertical and non-zero. We distinguish the cases $k >m$ and $k =m$.

Suppose $k >m$. Then $\Ker (L : \wedge ^{k-2} \frak n _1 ^* \to \wedge ^k \frak n_1 ^*)$ is non-zero. Let $\alpha \in \Omega  _1 ^{k-2}$ be non-zero, compactly supported and such that $\alpha \wedge d\tau =0$. For every $f \in \Omega _c ^0 (N)$ consider the form $\varphi = \varphi _2 \wedge \tau$, with $\varphi _2 := f \cdot \alpha$. Then $d \varphi$ is vertical. Moreover $d\varphi \neq 0$ for generic $f$.

Assume now that $k =m$. Then the Lefschetz map $L : \wedge ^{m-2} \frak n _1 ^* \to \wedge ^m \frak n_1 ^*$ is an isomorphism. Pick any compactly supported  $\varphi _1 \in \Omega _1 ^{m-1}$. Let $\varphi _2 \in \Omega _1 ^{m-2}$ be the unique solution of the equation: 
\begin{equation}
\label{cx-eqn} 
(d\varphi _1)_1 = -(-1)^m  \varphi _2 \wedge d\tau.
\end{equation}
and set $\varphi := \varphi_1 + \varphi_2 \wedge \tau$. Then $d \varphi$ is vertical. We claim that for generic $\varphi_1$, one has $d\varphi \neq 0$.
Indeed suppose that $d\varphi =0$. Then there exists $\beta \in \Omega _c ^{k-1}(N)$ so that $\varphi = d \beta$. Thus $\varphi _1 = (d \beta)_1$. Let $S \subset N$ be a complete horizontal submanifold of dimension $m-1$ (\textit{e.g.}\! the boundary at infinity of an isometric copy of $\mathbb H ^m _{\mathbf R}$ in $\mathbb H^m _{\mathbf C}$). 
Since $S$ is horizontal, one has by Stokes' Theorem:
$$\int _S \varphi _1 = \int _S (d\beta)_1 = \int _S d\beta = 0.$$
The claim follows.
\end{proof}

We can now conclude:

\begin{proof}[End of proof of Theorem \ref{complex-thm}]
Note that we can use Poincar\'e duality (Proposition \ref{dR-Pd}(2)) since we know at this stage that the cohomology spaces we consider are Hausdorff. 
The vanishing results in Items (2) and (3) follow from Lemma \ref{complex-lem1} and Poincar\'e duality. 
The non-vanishing ones are consequence of Lemma \ref{complex-lem2} and Poincar\'e duality.
\end{proof}

\subsection{Complement (on density) to Theorem \ref{complex-thm}}\label{complex-complement}
We establish two results (Propsitions \ref{null-prop} and \ref{dense-prop}) that complement Theorem \ref{complex-thm} and that could be useful in the future. The first one will serve partially in Section \ref{s - sl3} to study the cohomology of $\mathrm{SL}_3 (\mathbf R)/{\rm SO}_3({\mathbf R})$. The objects and notations are the same as in the previous section.

Recall from Theorem \ref{complex-thm}, that for $k \in \{1, \dots, m\}$ and $\frac{2m}{k} <p< \frac{2m}{k-1}$, the space $L^p \mathrm{H}_{\mathrm {dR}} ^k (R)$
is Hausdorff, non-zero, and admits the subspace $\{[d(\chi \cdot \pi ^* \theta)] : \theta \in \Omega _c ^{k-1} (N)\}$ as a dense subset.
The first result of the section describes the classes $[d(\chi \cdot \pi ^* \theta)]$ that are null in $L^p \mathrm{H}_{\mathrm {dR}} ^k (R)$: 

\begin{proposition}\label{null-prop}
Let $k \in \{1, \dots, m\}$ and $\frac{2m}{k} <p< \frac{2m}{k-1}$. For every $\theta \in \Omega _c ^{k-1}(N)$, the following holds:
\begin{enumerate}
\item When $k<m$, the class $[d(\chi \cdot \pi ^* \theta)]$ is null in $L^p \mathrm{H}_{\mathrm {dR}} ^k (R)$ if and only if $(d\theta)_1 = \gamma \wedge d\tau$, for some horizontal form $\gamma \in \Omega _c ^{k-2}(N)$.
\item  When $k=m$, the class $[d(\chi \cdot \pi ^* \theta)]$ is null in $L^p \mathrm{H}_{\mathrm {dR}} ^k (R)$ if and only if
$$d\Bigl( \theta - \mathcal L \bigl((d\theta)_1 \bigr) \wedge \tau \Bigr) =0,$$
where $\mathcal L : \Omega _1 ^{m} \to \Omega _1 ^{m-2}$ denotes the pointwise operator induced by the inverse of the Lefschetz isomorphism
$L_{m-2}: \wedge ^{m-2} \frak n _1^* \to \wedge ^{m} \frak n _1^*$.
\end{enumerate}
\end{proposition}

\begin{proof} 
(1). Let $\theta \in \Omega _c ^{k-1}(N)$ and suppose that $d\theta$ can be written $d\theta = \gamma \wedge d\tau + \delta \wedge \tau$, with $\gamma$ and $\delta$ horizontal. When $k=1$, such a relation is impossible unless $\theta =0$. Namely the differential of a non-zero compactly supported function has always a non-zero horizontal component. Let us assume $k \geqslant 2$. We claim that there exists a horizontal form $\beta \in \Omega _c ^{k-2}(N)$, such that $d\theta = d(\beta \wedge \tau)$. In combination with Corollary \ref{boundary-cor} and Lemma \ref{complex-lem1}, this yields $[d(\chi \cdot \pi ^* \theta)] =0$.

Since $\beta, \gamma$ and $\delta$ are horizontal forms, the equation $d\theta = d(\beta \wedge \tau)$
is equivalent to the following system of two equations 
$$ \gamma \wedge d\tau = (-1)^k \beta \wedge d\tau \quad \mathrm{and} \quad \delta \wedge \tau = d\beta \wedge \tau.$$
Set $\beta := (-1)^k \gamma$. Then the first equation is satisfied. Moreover one has $d\beta \wedge \tau = (-1)^k d\gamma \wedge \tau = (-1)^k (d\gamma)_1 \wedge \tau$. Thus the second equation is satisfied if the relation $\delta = (-1)^k (d\gamma)_1$ holds. Since $dd\theta =0$, one has $d\gamma \wedge d\tau + d\delta \wedge \tau - (-1)^{k} \delta \wedge d\tau=0$. This implies that $((d \gamma)_1 - (-1)^{k}\delta)\wedge d\tau =0$, which in turn implies that $(d \gamma)_1 - (-1)^{k}\delta=0$ since the Lefschetz map $L_{k-1}$ is injective (recall that $k<m$ by assumption). Therefore the second equation is satisfied and the claim is proved.

Conversely, let $\theta \in \Omega _c ^{k-1}(N)$ be such that $[d(\chi \cdot \pi^*\theta)] =0$. Denote by $(q,\ell)$ the Poincar\'e dual of $(p,k)$ relatively to $R$. One has $\frac{2m}{\ell +1} <q< \frac{2m}{\ell}$ and $\ell >m$. In particular $L_{\ell -2}$ is not injective, and therefore the following subspace is non-trivial
$$\Gamma = \{\alpha \in \Omega _c ^{\ell -2}(N) : \alpha \mathrm{~is~horizontal~and~} \alpha \wedge d\tau =0\}.$$
Pick $\alpha \in \Gamma$ and consider the form $\varphi = \alpha \wedge \tau \in \Omega _c ^{\ell -1}(N)$. 
One has $d\varphi = d\alpha \wedge \tau +(-1)^\ell \alpha \wedge d\tau = d\alpha \wedge \tau$. 
Thus $d\varphi$ is vertical, and so by Lemma \ref{complex-lem2}(2) the form $\pi^*(d\varphi)$ belongs to $\mathcal Z^{q, \ell}(R, \xi)$. 
Our assumption $[d(\chi \cdot \pi^*\theta)] =0$, in combination with Theorem \ref{Lie-thm}(4), implies that $\int _N d\varphi \wedge \theta =0$. With Stokes' formula and the definition of $\varphi$, it follows that $\int_N \alpha \wedge \tau \wedge d\theta =0$, {\it i.e.}\! $\int_N \alpha \wedge \tau \wedge (d\theta)_1 =0$. 
So far we have established that every $\theta \in \Omega _c ^{k-1}(N)$ such that $[d(\chi \cdot \pi^*\theta)] =0$ satisfies the following property:
\begin{equation}\label{propertybis-complex} \int_N \alpha \wedge \tau \wedge (d\theta)_1 =0 \mathrm{~~for~all~}\alpha \in \Gamma.
\end{equation}
When $k \geqslant 2$, we will show that Property (\ref{propertybis-complex}) implies that $(d\theta)_1$ can be written $\gamma \wedge d\tau$. When  $k=1$, we will prove that (\ref{propertybis-complex}) implies $(d\theta)_1 =0$.

Suppose first that $k=1$. One has $\ell = 2m-1$, $L_{\ell-2} =0$ and $\Gamma = \Omega _1^{\ell -2}\cap \Omega _c ^{\ell -2}(N)$. Thus Property (\ref{propertybis-complex}) yields that $\int_N \omega \wedge \tau \wedge (d\theta)_1 =0$ for every $\omega \in \Omega _c ^{\ell -2}(N)$. This in turn implies that $\tau \wedge (d\theta)_1 =0$, {\it i.e. \!} $(d\theta)_1=0$.

Suppose now that $k\geqslant 2$. We will use the following lemma:

\begin{lemma}
\label{key-lemma} 
Let $b : \wedge ^{\ell -2} \frak n _1 ^* \times \wedge ^k \frak n _1 ^* \to \mathbf R$, be the non-degenerate bilinear form defined by $b(u,v) = u \wedge v$. Relative to $b$, one has $(\Ker L_{\ell -2})^\bot = {\rm Im} \, L_{k-2}$.
\end{lemma}

\begin{proof}
[Proof of Lemma \ref{key-lemma}]
Since $b$ is non-degenerate, the statement is equivalent to $({\rm Im} \, L_{k-2} )^\bot = \Ker L_{\ell -2}$.
Let $u \in \wedge ^{\ell -2}\frak n_1^*$. 
It belongs to $({\rm Im} \, L_{k-2} )^\bot$ if and only if $u \wedge v \wedge d\tau =0$ for all $v \in \wedge ^{k -2}\frak n_1^*$.  This is equivalent to $u \wedge d\tau =0$, {\it i.e.}\! to $u \in \Ker L_{\ell -2}$.
\end{proof}

Lemma \ref{key-lemma} allows one to complete the proof of Item (1) as follows. By definition, $\Gamma$ is the space of compactly supported smooth sections of the left-invariant vector bundle over $N$ generated by $\Ker L_{\ell -2}$. Property (\ref{propertybis-complex}) can be interpreted as saying that $(d\theta)_1$ is a smooth section of the left-invariant vector bundle generated by $(\Ker L_{\ell -2})^\bot$. By Lemma \ref{key-lemma}, this is equivalent to $(d\theta)_1 = \gamma \wedge d\tau$ for some horizontal form $\gamma \in \Omega _c ^{k-2}(N)$.

(2). Suppose that $k=m$, and let $(q,m)$ be the Poincar\'e dual of $(p,m)$ relatively to $R$. One has $\frac{2m}{m+1} <q< \frac{2m}{m}=2$. Let $\theta \in \Omega _c ^{m-1}(N)$. According to Theorem \ref{Lie-thm} and Poincar\'e duality (Proposition \ref{dR-Pd}), the class $[d(\chi \cdot \pi^*\theta)]$ vanishes in $L^p \mathrm{H}_{\mathrm {dR}} ^m (R)$ if and only if $T(\theta) =0$ for all $T \in \mathcal D '^m(N)$ such that $\pi^*T \in \mathcal Z^{q,m}(R, \xi)$.

Since $L_{m-2}$ is an isomorphism, there exist unique horizontal forms $\alpha \in \Omega _c ^{m-2}(N)$ and $\beta \in \Omega _c ^{m-1}(N)$, such that $d\theta$ weight decomposes as $d\theta = \alpha \wedge d\tau + \beta \wedge \tau$.

Let $T \in \mathcal D '^m(N)$ be such that $\pi^*T \in \mathcal Z^{q,m}(R, \xi)$. Since $N$ is contractible and $T$ is closed, it admits a primitive, say $S \in \mathcal D '^{m-1}(N)$. Then $T(\theta)$ admits the following expression:

\begin{lemma} \label{lemma-complement} With the notation above, one has
$$T(\theta) = (-1)^m S_1 \Bigl( \bigl(\beta  - (-1)^m d\alpha \bigr) \wedge \tau \Bigr).$$
Moreover:
$$\bigl( \beta - (-1)^m d\alpha \bigr) \wedge \tau  = d\Bigl(\theta -(-1)^m \mathcal L \bigl((d\theta)_1\bigr) \wedge \tau \Bigr).$$
\end{lemma}

Assume for a moment that the lemma holds. Then the ``if'' part of item (2) follows immediately. To establish the ``only if'' part, we apply the lemma with some explicit currents $T$. Let $\varphi _1 \in \Omega _c ^{m-1}(N)$ be an arbitrary horizontal form, let $\varphi _2 \in \Omega _c ^{m-2}(N)$ be the horizontal form uniquely determined by the equation $(d\varphi_1)_1 = -(-1)^m \varphi _2 \wedge d\tau$. Set $\varphi := \varphi _1 + \varphi _2 \wedge \tau$. Then an easy computation shows that $d\varphi$ is vertical. Thus, by Lemma \ref{complex-lem2}(2), the current $\pi ^* (d\varphi)$ belongs to $\mathcal Z^{q,m}(R, \xi)$. Lemma \ref{lemma-complement} applied to $T = d\varphi$, yields that
$$\int _N \varphi_1 \wedge \bigl(\beta - (-1)^m d\alpha \bigr) \wedge \tau =0,$$
for every horizontal $\varphi _1 \in \Omega _c ^{m-1}(N)$.
Therefore $\bigl(\beta - (-1)^m d\alpha \bigr) \wedge \tau =0$, and the second part of the lemma completes the proof of item (2).
\end{proof}

It remains to give the 

\begin{proof}[Proof of Lemma \ref{lemma-complement}]

Since $S$ is a primitive of $T$, we have $T(\theta) = dS(\theta) = (-1)^m S(d\theta)$. 
From Relation (\ref{weight}) and the expression $d\theta =  \alpha \wedge d\tau + \beta \wedge \tau$, it follows that $S(d\theta) = S_1 (\beta \wedge \tau) + S_2 (\tau \wedge \alpha \wedge d\tau)$. By Lemma \ref{complex-lem2}(1) the current $T$ is vertical. This means (by a simple computation) that $(dS_1)_1 = -(-1)^m S_2 \wedge d\tau$. Therefore:
\begin{align*}
S_2 (\tau \wedge \alpha \wedge d\tau) &= (-1)^m S_2 (d\tau \wedge \alpha \wedge \tau)= -(dS_1)_1(\alpha \wedge \tau) \\
&= -dS_1(\alpha \wedge \tau) = -(-1)^m S_1 \bigl(d(\alpha \wedge \tau)\bigr) \\ 
&= S_1 \bigl(-(-1)^m d\alpha \wedge \tau\ \bigr).
\end{align*}
The expected formula for $T(\theta)$ follows. To establish the second formula, we compute
\begin{align*}
d\Bigl(\theta -(-1)^m \mathcal L \bigl((d\theta)_1 \bigr) \wedge \tau \Bigr) &= d\bigl( \theta -(-1)^m \alpha \wedge \tau \bigr) \\
&= \alpha \wedge d\tau + \beta \wedge \tau -(-1)^m d\alpha \wedge \tau - \alpha \wedge d\tau \\
&= \bigl( \beta - (-1)^m d\alpha \bigr) \wedge \tau.
\end{align*}
The lemma is proved.
\end{proof}

The second result of the section deals with the Banach space $\mathcal Z^{p,k}(R, \xi)$ for $k \in\{m, 2m-1\}$ and $\frac{2m}{k+1} <p< \frac{2m}{k}$. 
According to Lemma \ref{complex-lem2}(2), the forms $\pi^*(d\varphi)$, where $\varphi \in \Omega_c ^{k-1}(N)$ and $d\varphi$ is vertical, belong to $\mathcal Z^{p,k}(R, \xi)$. 
A natural problem is to determine whether they form a dense subspace. For the norm topology, we do not know, but for the current topology this is indeed the case:

\begin{proposition}\label{dense-prop} Let $k \in\{m, \dots, 2m-1\}$ and $\frac{2m}{k+1} <p< \frac{2m}{k}$. The set $\{d\varphi: \varphi \in \Omega_c ^{k-1}(N), ~d\varphi \mathrm{~is~vertical}\}$ is a dense subspace in the sense of currents in $\{T \in
\mathcal D '^k(N) : \pi^* T \in \mathcal Z^{p,k}(R, \xi)\}$.
\end{proposition}

\begin{proof} Set $F := \{d\varphi: \varphi \in \Omega_c ^{k-1}(N), ~d\varphi \mathrm{~is~vertical}\}$ and $E:= \{T \in
\mathcal D '^k(N) : \pi^* T \in \mathcal Z^{p,k}(R, \xi)\}$ for simplicity.
Let $(q, \ell)$ be the Poincar\'e dual of $(p,k)$ relatively to $R$. The topology on $E$, which is induced by the weak*-topology of $\mathcal D'^k (N)$, is generated by the linear forms $\Lambda _\theta : E \to \mathbf R$ defined by $\Lambda _\theta (T) =T(\theta)$, where $\theta$ belongs to $\Omega _c ^{\ell -1}(N)$. We shall let $E_{\mathrm{top}}$ denote $E$ equipped with this topology.

According to the Hahn-Banach Theorem, showing that $F$ is dense in $E_{\mathrm{top}}$, is equivalent to proving the triviality of every $\Lambda \in E_{\mathrm{top}}^*$ such that $\Lambda (F) =\{0\}$.

Since every element of $E_{\mathrm{top}}^*$ is a $\Lambda _\theta$ for some $\theta \in \Omega _c ^{\ell -1}(N)$ (see \cite[Theorem 3.10]{Ru}),  we are led to showing the following: if $\theta \in \Omega _c ^{\ell-1}(N)$ satisfies $\int _N d\varphi \wedge \theta = 0$ for every $d\varphi \in F$, then $\Lambda _\theta = 0$.

By analysing the ``only if'' parts of the proof of the previous Proposition \ref{null-prop}, one sees that the $\theta$'s such that $\int _N d\varphi \wedge \theta = 0$ for every $d\varphi \in F$, are precisely those for which  $[d(\chi \cdot \pi ^* \theta)] =0$ in $L^q \mathrm{H}_{\mathrm{dR}} ^\ell (R)$. Therefore they satisfy $\Lambda _\theta =0$, thanks to Theorem \ref{Lie-thm}(4).
\end{proof}

\section{The symmetric space ${\rm SL}_3({\mathbf R})/{\rm SO}_3({\mathbf R})$} 
\label{s - sl3}

We prove Theorem \ref{SL3-theorem} (stated in the introduction) which describes the second $L^p$-cohomology of ${\rm SL}_3({\mathbf R})/{\rm SO}_3({\mathbf R})$. The strategy is similar to the one conducted in Section \ref{s - non QI solvable} to study the second $L^p$-cohomology of the groups $S_\alpha$. It highly relies on the description of the cohomology of the complex hyperbolic plane discussed in Section \ref{s - complex hyp sp}.

\subsection{Notation and decomposition of ${\rm SL}_3({\mathbf R})$}\label{SL3-notation}
At first we introduce the various subgroups of $\mathrm{SL}_3(\mathbf R)$ we will be working with in the sequel.

The relevant  Iwasawa decomposition here is $\mathrm{SL}_3(\mathbf R) = KAN$, with $K=\mathrm{SO}_3(\mathbf R)$, $A = \mathrm{Diag}(\mathbf R ^3) \cap \mathrm{SL}_3(\mathbf R)$ and 
$$N = \Bigl\{\left(
        \begin{array}{ccc}
          1 & x & z \\
          0 & 1 & y \\
          0 & 0 & 1
        \end{array}
      \right) : x, y, z \in \mathbf R \Bigr\}
   \simeq \mathrm{Heis}(3).$$
Let $\frak a$ and $\frak n$ be the Lie algebras of $A$ and $N$ respectively. Every element of $\frak n$ can be naturally denoted by a triple $(x,y,z) \in \mathbf R ^3$.

Let $\xi, \eta \in \frak a$ be defined by $\xi = \mathrm{diag}(-1,0,1)$ and $\eta = \frac{1}{3} \mathrm{diag}(1,-2,1)$. They act on $\frak n$ by: 
\begin{equation}\label{notation-equation}
\mathrm{ad} \xi \cdot (x,y,z) = (-x, -y, -2z) ~~\mathrm{and}~~ 
\mathrm{ad} \eta \cdot (x,y,z) = (x, -y, 0). 
\end{equation}
We consider the following subgroups of  $\mathrm{SL}_3 (\mathbf R)$:$$R:=\{e^{t\xi}\}_{t \in \mathbf R} \ltimes N \simeq \mathbb H ^2_{\mathbf C},\quad  H := \{e^{s\eta}\}_{s \in \mathbf R} \ltimes N \quad \mathrm{~and}$$
$$S := A \ltimes N = \{e^{s\eta}\}_{s \in \mathbf R} \ltimes R = \{e^{t\xi}\}_{t \in \mathbf R} \ltimes H.$$
The Lie group $S$ is isometric to the symmetric space $\mathrm{SL}_3 (\mathbf R) / \mathrm{SO}_3(\mathbf R)$.

\subsection{First observations} \label{SL3-obs} 
We derive from previous results some preliminary observations on $L^p \mathrm{H}_{\mathrm{dR}}^2 (S)$. 
The notations are the same as in the previous section.

\begin{proposition}\label{SL3-obs1} One has $L^p \mathrm{H}_{\mathrm{dR}}^2 (S) = \{0\}$ for $p < \frac{4}{3}$.
\end{proposition}

\begin{proof} Let $\frak h$ denotes the Lie algebra of $H$. Set $\delta := \mathrm{ad} \xi \restr _{\frak h} \in \mathrm{Der}(\frak h)$, so that $S$ can be written $S = \mathbf R \ltimes _\delta H$. The ordered list of eigenvalues of $-\delta$ enumerated with multiplicity, is
$$\lambda _1 =0< \lambda _2 = \lambda _3 =1 <\lambda _4=2.$$
Thus, with the notation of Section \ref{s - Lie}, the trace of $-\delta$ is $h=4$, and one has $W _{2} = \lambda _3 + \lambda _4=3$. Therefore the statement follows from Theorem \ref{Lie-thm}(1).
\end{proof}

 \begin{proposition}\label{SL3-obs2} For $p \in (\frac{4}{3}; 4)\setminus \{2\}$, the space $L^p \mathrm{H}_{\mathrm{dR}}^2 (R)$ is non-zero and Hausdorff, and there exists a linear isomorphism
$$L^p \mathrm{H}_{\mathrm{dR}}^2 (S) \simeq \Bigl\{[\omega ] \in L^p \mathrm{H}_{\mathrm{dR}}^2 (R) : \int _{\mathbf R} \Vert {e^{s\mathrm{ad}\eta}}^* [\omega] \Vert^p _{L^p \mathrm{H}^ 2(R)} ds < +\infty \Bigr\}.$$
\end{proposition}

\begin{proof} When $p \in (\frac{4}{3}; 4)\setminus \{2\}$, Theorem \ref{complex-thm} shows that $L^p \mathrm{H}_{\mathrm{dR}}^2 (R)$ is non-zero and Hausdorff, and that $L^p \mathrm{H}_{\mathrm{dR}}^k (R) = \{0\}$ in all degrees $k\neq 2$. Since $S = \{e^{s\eta}\}_{s \in \mathbf R} \ltimes R$, the above description of the cohomology of $R$, in combination with a Hochschild-Serre spectral sequence argument (see \cite[Corollary 6.10]{BR2}), yields the desired linear isomorphism.
\end{proof}

\begin{proposition}\label{SL3-obs3} For $p>4$, the space $L^p \overline{\mathrm{H}_{\mathrm{dR}}^2} (S)$ is non-zero. 
\end{proposition}

\begin{proof} Consider again $\lambda _1 =0< \lambda _2 = \lambda _3 =1 <\lambda _4=2$ the list of the eigenvalues of $-\delta = - \mathrm{ad} \xi \restr _{\frak h} \in \mathrm{Der}(\frak h)$. The trace of $-\delta$ is $h=4$, and one has $w _{2} = \lambda _{1} + \lambda _{2} = 1$. 
Since the rank of $S$ is equal to $2$, it follows from \cite[Theorem C and Corollary 3.4]{BR2} that $L^p \overline{\mathrm{H}_{\mathrm{dR}}^2} (S)$ is non-zero for $p>\frac{h}{w _{2}}=4$.
\end{proof}

\subsection{Auxiliary results on $\mathbb H ^2 _{\mathbf C}$}\label{SL3-aux}
This section is devoted to auxiliary results (Lemmata \ref{SL3-aux1} and \ref{SL3-aux2}) that will serve in the next section to prove Theorem \ref{SL3-theorem}. 
 
Recall that $R=\{e^{t\xi}\}_{t \in \mathbf R} \ltimes N \simeq \mathbb H ^2_{\mathbf C}$. Let $\pi _N$ be the projection map from $R$ onto $N$.
 
 \begin{lemma}\label{SL3-aux1}
 Suppose that $p \in (2; 4)$ and let $\theta \in \Omega _c ^1 (N) \setminus \{0\}$ be of the form $\theta = fdx$ or $gdy$. Then: 
 \begin{enumerate}
  \item The class $[d(\chi \cdot \pi ^*_N \theta)]$ is non-zero in $L^p \mathrm{H}_{\mathrm{dR}}^2 (R)$.
  \item If $\theta = fdx$ (resp. \!$gdy$), one has 
  $$\bigl\Vert [d(\chi \cdot \pi ^* _N {e ^{s \mathrm{ad}\eta}} ^* \theta)] \bigr\Vert _{L ^p \mathrm H ^{2}(R)} \to 0$$
  exponentially fast, when $s$ tends to $-\infty$ (resp. \!$+\infty$).
 \end{enumerate}
 \end{lemma}

\begin{proof}
(1). We apply the criterion in Proposition \ref{null-prop}(2). Suppose first that $\theta = fdx$. With the notation of Sections \ref{complex-diff} and \ref{complex-complement}, one has $d\theta = (Y\cdot f)dy \wedge dx + (Z \cdot f) \tau \wedge dx$. Since $d\tau =-dx \wedge dy$, we get that 
\begin{align*}
d\Bigl( \theta - \mathcal L \bigl((d\theta) _{1}\bigr)\tau \Bigr) &= d\bigl(\theta -(Y \cdot f)\tau\bigr)\\
&=-(Z \cdot f + X\cdot Y \cdot f) dx \wedge \tau - (Y^{2} \cdot f)dy \wedge \tau.
\end{align*}
The term $ Y^{2} \cdot f$ is non-zero since the function $f$ is non-zero and has compact support. Therefore $d( \theta - \mathcal L ((d\theta) _{1}\tau)$ is non-zero, and the statement follows from Proposition \ref{null-prop}(2). The case $\theta = gdy$ is similar.

(2). Suppose that $\theta = fdx$. For $(s,t) \in \mathbf R ^{2}$, one has
$${e^{t\mathrm{ad} \xi}}^* ({e^{s\mathrm{ad} \eta}}^* \theta) = (e ^{t\mathrm{ad} \xi +s \mathrm{ad} \eta})^*\theta
\mathrm{~~and~~}{e^{t\mathrm{ad} \xi}}^* (d {e^{s\mathrm{ad} \eta}}^* \theta) = (e ^{t\mathrm{ad} \xi +s \mathrm{ad} \eta})^*d\theta.$$
By Lemma \ref{Lie-lemma} and relation (\ref{notation-equation}), their $L^{p}$-norms satisfy
$$\bigl\Vert (e ^{t\mathrm{ad} \xi +s \mathrm{ad} \eta})^*\theta \bigr\Vert _{L^{p}\Omega ^{1}(N)} \asymp e^{(\frac{4}{p} -1)t + s} \Vert f \Vert _p,$$
$$\mathrm{~~and~~}\bigl\Vert (e ^{t\mathrm{ad} \xi +s \mathrm{ad} \eta})^*d\theta \bigr\Vert _{L^{p}\Omega ^{2}(N)} \asymp e^{-(2-\frac{4}{p})t} \Vert Y\cdot f\Vert _p + e^{-(3-\frac{4}{p})t +s} \Vert Z\cdot f\Vert _p.$$
Set $a = \frac{4}{p}-1$, $b= 2- \frac{4}{p}$ and $c= 3- \frac{4}{p}$. One has $a, b, c >0$ since $2<p<4$.
From Proposition \ref{Lie-inequalities}, it follows that
$$\bigl\Vert [d(\chi \cdot \pi ^* _N {e ^{s \mathrm{ad}\eta}} ^* \theta)] \bigr\Vert _{L ^p \mathrm H ^{2}(R)} \leqslant C \inf _{t \in \mathbf R} \bigl\{ (e^{at} + e^{-ct})e^{s} + e^{-bt} \bigr\}.$$
Suppose that $s \to -\infty$, and set $t= -\frac{s}{2a}$. Then
$$(e^{at} + e^{-ct})e^{s} + e^{-bt} = e^{\frac{s}{2}} + e^{(\frac{c}{2a} +1)s} + e^{\frac{b}{2a}s}, $$
which tends to $0$ exponentially fast. The case $\theta = gdy$ is similar.
\end{proof}

 \begin{lemma}\label{SL3-aux2}
Suppose that $p \in (2; 4)$. There exist non-zero forms $\theta =fdx$ and $\Theta =gdy$ in $\Omega _c ^{1}(N)$, such that $[d(\chi \cdot \pi _N ^* \theta)] = [d(\chi \cdot \pi _N ^* \Theta)]$ in $L^p \mathrm{H}_{\mathrm{dR}}^2 (R)$.
\end{lemma}

\begin{proof} Let $u$ be an arbitrarily non-zero function in $\Omega _c ^0 (N)$. Its differential is $(X\cdot u)dx+ (Y\cdot u)dy + (Z\cdot u) \tau$. Set $f:= X\cdot u$, $g:= -Y\cdot u$, $h := -Z\cdot u$, and let $\theta =fdx$ and $\Theta =gdy$. By Proposition \ref{null-prop}(2), one has $[d(\chi \cdot \pi _N ^* \theta)] = [d(\chi \cdot \pi _N ^* \Theta)]$ in $L^p \mathrm{H}_{\mathrm{dR}}^2 (R)$. Indeed: 
$$d(\theta - \Theta) = d(h\tau) = dh\wedge \tau + hd\tau,$$
thus $\mathcal L ((\theta -\Theta ) _{1}) = h$, and we get that $d(\theta - \Theta - \mathcal L ((\theta -\Theta ) _{1}) \tau ) =0$.
\end{proof}

 \subsection{Proof of Theorem \ref{SL3-theorem}}\label{SL3-proof}

Thanks to Propositions \ref{SL3-obs1} and \ref{SL3-obs3}, we can restrict ourselves to the region $p \in (\frac{4}{3}; 4)$. In this region, Proposition \ref{SL3-obs2} shows that there is a linear isomorphism 
\begin{equation}\label{SL3-iso}
 L^p \mathrm{H}_{\mathrm{dR}}^2 (S) \simeq \Bigl\{[\omega ] \in L^p \mathrm{H}_{\mathrm{dR}}^2 (R) : \int _{\mathbf R} \Vert {e^{s\mathrm{ad}\eta}}^* [\omega] \Vert^p _{L^p \mathrm{H}^ 2(R)} ds < \infty \Bigr\}.
\end{equation}
We will use this representation to prove the theorem.

{\it Step 1}. $L^p \mathrm{H}_{\mathrm{dR}}^2 (S) = \{0\}$ when $p \in (\frac{4}{3}; 2)$.

Let $p \in (\frac{4}{3}, 2)$ and let $(q, 2)$ be the Poincar\'e dual of $(p, 2)$ relatively to $R$. One has $q \in (2; 4)$. According to the relation (\ref{SL3-iso}), it is enough to show that for every non-trivial $[\omega]$, one has $\Vert {e^{s\mathrm{ad}\eta}}^* [\omega] \Vert_{L^p \mathrm{H}^ 2(R)} \to +\infty$, either when $s$ tends to $+\infty$ or to $-\infty$.

So let $[\omega]$ be a non-trivial class in $L^p \mathrm{H}_{\mathrm{dR}}^2 (R)$. By Theorem \ref{Lie-thm}(4), it admits a boundary value $T \in {\mathcal D’} ^{2}(N) \cap \Ker d$, so that
$$\Vert [\omega] \Vert_{L^p \mathrm{H}^ 2(R)} = \sup \Bigl\{T(\theta) : \theta \in \Omega _c ^{1}(N) , \bigl\Vert [d(\chi \cdot \pi ^* _N \theta )] \bigr\Vert _{L^q \mathrm{H}^ 2(R)} \leqslant 1\Bigr\}. $$
In the group $R$, right multiplication by $\exp t \xi$ commutes with conjugacy by $\exp s \eta$. Therefore the boundary value of the class ${e ^{s \mathrm{ad} \eta}}^* [\omega]$ is the current ${e ^{s \mathrm{ad} \eta}}^* T$.
With a change of variable, one gets
\begin{align*}
\Vert {e ^{s \mathrm{ad} \eta}}^*[\omega] \Vert_{L^p \mathrm{H}^ 2(R)} = \sup \Bigl\{&T(\theta) : \theta \in \Omega _c ^{1}(N), \mathrm{~with}\\ &\bigl\Vert [d(\chi \cdot \pi ^* _N {e ^{s \mathrm{ad} \eta}}^*\theta )] \bigr\Vert _{L^q \mathrm{H}^ 2(R)} \leqslant 1\Bigr\}.
\end{align*}
By Lemma \ref{complex-lem2}, the current $T$ is vertical, thus it can be written as $T = F dy \wedge \tau + G dx \wedge \tau$, with $F, G \in {\mathcal D’} ^{0}(N)$. If $F \neq 0$ (resp. \!$G \neq 0$), then for $\theta = fdx \in \Omega _c ^{1} (N)$ (resp. \!$\theta = gdy$), one has $T(\theta) = F(f \mathrm{vol})$ (resp. \!$-G(g \mathrm{vol})$). In any case, there exists $\theta  \in \Omega _c ^1 (N)$, of the form $fdx$ or $gdy$, such that $T(\theta )=1$.
The above equality in combination with Lemma \ref{SL3-aux1}(2), yields that
$$\Vert {e ^{s \mathrm{ad} \eta}}^*[\omega] \Vert_{L^p \mathrm{H}^ 2(R)} \geqslant \bigl\Vert [d(\chi \cdot \pi ^* _N {e ^{s \mathrm{ad} \eta}}^*\theta )] \bigr\Vert _{L^q \mathrm{H}^ 2(R)} ^{-1} \to +\infty, $$
either when $s$ tends to $+\infty$ or to $-\infty$.

{\it Step 2}. $L^p \mathrm{H}_{\mathrm{dR}}^2 (S) \neq \{0\}$ when $p \in (2; 4)$.
 
Let $p \in (2; 4)$. We will exhibit some non-trivial element in the right-hand side of (\ref{SL3-iso}).

Let $\theta =fdx$ and $\Theta =gdy$ be as in Lemma \ref{SL3-aux2}. Set $\omega := d(\chi \cdot \pi _N ^* \theta)$ and $\Omega := d(\chi \cdot \pi _N ^* \Theta)$. By Lemmata \ref{SL3-aux2} and \ref{SL3-aux1}(1), their classes $[\omega]$ and $[\Omega]$ are equal and non-zero in $L^p \mathrm{H}_{\mathrm{dR}}^2 (R)$. 

Since $\xi$ and $\eta$ commute, one has ${e^{s\mathrm{ad} \eta}}^* \omega = d(\chi \cdot \pi _N ^* {e^{s\mathrm{ad} \eta }}^* \theta)$.
Thus by Lemma \ref{SL3-aux1}(2), the norm $\Vert {e^{s\mathrm{ad}\eta }}^*[\omega] \Vert _{L^p \mathrm{H}^ 2(R)}$ tends to $0$ exponentially fast when $s$ tends to $-\infty$, and similarly for $\Vert {e^{s\mathrm{ad}\eta}}^* [\Omega] \Vert _{L^p \mathrm{H}^ 2(S)}$ when $s$ tends to $+\infty$. Since $[\omega] =[\Omega]$, the integral $\int _{\mathbf R} \Vert {e^{s\mathrm{ad}\eta}}^* [\omega] \Vert^p _{L^p \mathrm{H}^ 2(R)} ds $ converges. Thus $[\omega]$ provides a non-trivial element in the right-hand side of (\ref{SL3-iso}). 

\appendix

\section{The groups $S _\alpha \in \mathcal S ^{r, n}$: basic properties}\label{basic-sec}

Let $r , n \in \mathbf N$. We consider here the solvable Lie groups of the form $S_\alpha = \mathbf R^r \ltimes _\alpha \mathbf R^n$, where $\alpha : \mathbf R^r \to \{\mathrm{diagonal~~automorphisms~~of~~}\mathbf R^n\}$ is a Lie group morphism. We denote by $\varpi _i \in (\mathbf R^r)^*$ ($i = 1, \dots, n$) the {\it weights} associated to $\alpha$, \emph{i.e.\!} the linear forms so that
$\alpha = e ^{\mathrm{diag}(\varpi _{1}, \dots , \varpi _{n})}.$

Some groups $S _\alpha$ can be written with several couples of exponents $r, n$ -- {\it e.g.\!} when they are abelian. In the sequel we will always assume that the dimension $n$ of the second factor is minimal. This assumption is equivalent to require every weight $\varpi _i$ to be non-zero; and it forces the rank of $S _\alpha$ to be equal to $r$.

Let denote by $\mathcal S ^{r, n}$ the set of the groups $S_\alpha$ of exponents $r, n$ (with the above convention).

We remark that the groups $S_\alpha$ appear as special cases of the so-called \emph{abelian-by-abelian} solvable Lie groups. The latter ones are considered by Peng in \cite{Pen1, Pen2}. For those which are in addition unimodular, she establishes several quasi-isometric rigidity results. 

This Section is devoted to the following proposition, which establishes some of the basic properties of $S_\alpha$.
The first two are elementary. The third one follows from Azencott-Wilson’s classification of Lie groups of non-positive curvature \cite{AW} (see also \cite{Heb} for a detailed account of their geometric properties).

\begin{proposition}\label{appendix_prop} Let $S _\alpha \in \mathcal S ^{r, n}$. Then:

\begin{enumerate}

\item $S_\alpha$ admits {\bf no} non-trivial abelian direct factor if, and only if, the weights $\varpi _{1}, \dots , \varpi _{n}$ generates the vector space $(\mathbf R ^r)^*$.

\item Suppose that $S_\alpha$ has no abelian direct factor. Then $S_\alpha$ is reductible -- {\it i.e.\!} splits as a direct product of non-trivial closed subgroups -- if and only if there exists a non-trivial partition $I _{1} \sqcup I _{2} = \{1, \dots , n\}$
such that $$\mathrm{Span}\{\varpi _i : i \in I _{1}\} \cap \mathrm{Span}\{\varpi _i : i \in I _{2}\} = \{0\}.$$

\item $S_\alpha$ admits a left-invariant Riemannian metric of non-positive curvature if, and only if, $0$ does {\bf not} belong to the convex hull of $\{\varpi _{1}, \dots, \varpi _{n}\}$ in $(\mathbf R^r)^*$. In this case the factor $\mathbf R ^r$ is a totally geodesic Euclidean subspace of maximal dimension in $S _\alpha$.

\item There exists $\xi \in \mathbf R^r$, such that the associated subgroup $( \mathbf R \xi) \ltimes \mathbf R^n$ is quasi-isometric to the real hyperbolic space $\mathbb H ^{n+1}_{\mathbf R}$ if, and only if, the weights $\varpi _{1}, \dots, \varpi _{n}$ are contained in an affine subspace of $(\mathbf R^r)^*$, disjoint from $0$.
\end{enumerate}
\end{proposition}


It follows from Item (2) above, that the groups $S_\alpha$ which are irreducible, of higher rank, and of smallest dimension, belong to the family $\mathcal S^{2,3}$.

\begin{proof}
We start with some notations and preliminaries. Set $I:= \{1, \dots , n \}$ and denote by $(e _i)_{i \in I}$ the canonical basis of $\mathbf R ^{n}$.
The multiplication law of $S_\alpha$ is 
$$(u,x) \cdot (v,y) = \bigl(u+v, x+\alpha (u)y \bigr).$$
Let us denote the Lie algebra of $S_\alpha$ by $\frak s _\alpha$, and set $\Pi:= \mathrm{diag}(\varpi _{1}, \dots, \varpi _{n}) \in (\mathbf R ^{r})^* \otimes \mathrm{Diag}(\mathbf R ^{n})$, so that $\alpha = e^\Pi$. A standard computation gives the following expression for the Lie bracket in $\frak s _\alpha$:
\begin{equation}\label{bracket}
[(U,X), (V,Y)] = \bigl(0,\Pi(U)Y-\Pi(V)X \bigr).
\end{equation}
Since every $\varpi _i$ is non-zero, one checks that the center of $\frak s _\alpha$ is $$\frak z (\frak s _\alpha) = \{(U, 0) \in \frak s _\alpha ~\vert ~ \varpi _i (U) =0 \mathrm{~~for~~all~~}i\},$$ and its derived subalgebra is $[\frak s _\alpha , \frak s _\alpha]= \{0 _{\mathbf R ^r}\} \times \mathbf R ^n \simeq \mathbf R^n$.

\noindent{\it Proof of (1)} :
Suppose that the weights $\varpi _{1}, \dots , \varpi _{n}$ do not generate $(\mathbf R ^{r})^*$. Then $\Ker \Pi= \cap _{i \in I} \Ker \varpi _i$ is non-trivial. 
Write $\mathbf R ^{r} = \Ker \Pi \oplus E$, set $\frak h := \Ker \Pi \times \{0 _{\mathbf R ^{n}}\}$ and $\frak k := E \ltimes \mathbf R ^{n}$. 
Then $\frak h$ and $\frak k$ are supplementary subalgebras in $\frak s _\alpha$; moreover $\frak h$ is central, thus $S _\alpha$ admits a direct abelian factor.

Conversely, suppose that $S _\alpha $ admits a direct abelian factor. Then the center of $\frak s _\alpha$ is non-trivial. By the above description of $\frak z (\frak s _\alpha)$, there exists a non-zero $U \in \mathbf R ^r$ that belongs to the kernel of every $\varpi _i$. Therefore the $\varpi _i$’s do not generate $(\mathbf R ^{r})^*$.

\noindent{\it Proof of (2)} : Let $I = I_{1} \sqcup I_{2}$ be a non-trivial partition as in the statement. One has
$$\bigcap _{i \in I _{1}} \Ker (\varpi _i ) + \bigcap _{i \in I _{2}} \Ker (\varpi _i ) = \mathbf R ^r.$$
By assumption, every $\varpi _i$ is non-zero. Therefore $\mathbf R ^r$ admits proper supplementary subspaces $E_{1}, E_{2}$ such that for $j=1,2$:
$$\{0\} \neq E_j \subset \bigcap _{i \in I \setminus I_{j}} \Ker (\varpi _i ).$$
Set $\frak h_j:= E_j \ltimes \mathrm{Span}\{e_i : i \in I_j\}$. Then $\frak h_1$ and $\frak h_2$ are supplementary subalgebras in $\frak s _\alpha$. They satisfy $[\frak h _{1}, \frak h_{2}] =0$; indeed, for every $U \in E _{1}$ and $Y \in \mathrm{Span}\{e_i~;~i \in I_2\}$, one has
$$\Pi(U) Y = \sum _{i \in I _{2}} Y_{i} \varpi _i (U) e _i =0,$$
and similarly for $V \in E _{2}$ and $X \in \mathrm{Span}\{e_i : i \in I_1\}$. Thus $S_\alpha$ is reducible.

Conversely, suppose that $\frak s _\alpha$ admits supplementary (non-abelian) subalgebras $\frak h _{1}$, $\frak h _{2}$ such that 
$[\frak h _{1}, \frak h _{2}]=0$.
One has
$$[\frak h_{1}, \frak h_{1}] \cap [\frak h_{2}, \frak h_{2}] \subset \frak h_{1} \cap \frak h_{2}= \{0\}$$
$$\mathrm{and~~~}~[\frak s_\alpha , \frak s_\alpha] = [\frak h_{1} + \frak h_{2} , \frak h_{1} + \frak h_{2}] = [\frak h _{1}, \frak h _{1}] + [\frak h_{2}, \frak h_{2}].$$
Therefore $[\frak s_\alpha , \frak s_\alpha] = [\frak h _{1}, \frak h_{1}] \oplus [\frak h_{2}, \frak h_{2}]$,
which in turn implies that
\begin{equation}\label{commutator}
\frak h_{j} \cap [\frak s_\alpha , \frak s_\alpha] = [\frak h_{j}, \frak h_{j}],
\end{equation}
for $j=1,2$.
Let $\sigma: \mathbf R ^r \times \mathbf R^n \to \mathbf R^r$ be the projection map on the first factor, and set $E_j := \sigma (\frak h _j)$ for $j =1,2$. We claim that
\begin{equation}\label{h}
\frak h _j = E _j \ltimes [\frak h _j , \frak h _j].
\end{equation}
To see this, let $(U, X) \in \frak h _j$. One has $U = \sigma (U, X)$, thus $U \in E _j$. Since $\frak h _j$ is an ideal, relations (\ref{bracket}) and (\ref{commutator}) imply for every $V \in \mathbf R ^r$:
$$\bigl(0, -\Pi (V)X\bigr) = [(U, X), (V,0)] \in \frak h_{j} \cap [\frak s _\alpha , \frak s _\alpha] = [\frak h _{j}, \frak h _{j}].$$
By choosing $V$ such that $\Pi(V)$ is invertible, one has $X \in \Pi (V)^{-1}[\frak h _{j}, \frak h _{j}]$. On the other hand $\frak h _j$ is an ideal, thus so is $[\frak h _{j}, \frak h _{j}]$.  In combination with (\ref{bracket}), this implies that $\Pi (V) [\frak h _{j}, \frak h _{j}] = [\frak h _{j}, \frak h _{j}]$. Therefore $X \in [\frak h _{j}, \frak h _{j}]$, and claim (\ref{h}) follows now easily.

The subalgebras $\frak h _{1}$ and $\frak h _{2}$ are supplementary subspaces in $\frak s _\alpha$, thus so are $E _{1}$ and $E _{2}$ in $\mathbf R ^r$ -- thanks to (\ref{h}). Since $\frak h _j$ is non-abelian, $E _j$ is non-zero. Set $I _{1}:=\{i \in I : \varpi _i\restr _{E _{2}} =0\}$ and $I _{2}:=\{i \in I : \varpi _i\restr _{E _{1}} =0\}$. Observe that $I _{1} \cap I _{2} = \varnothing$ since every $\varpi _i$ is non-zero. We claim that $I _{1}$ and $I _{2}$ form a non-trivial partition of $I$. The proof of Item (2) follows then easily.
Let $(U, X) \in \frak h _{1}$ and $(V, Y) \in \frak h _{2}$. Since $[\frak h _{1}, \frak h _{2}]=0$, relation (\ref{bracket}) implies that $\Pi (U)Y = \Pi (V)X$. By (\ref{h}), one has $X \in [\frak h _{1}, \frak h _{1}]$ and $Y \in [\frak h _{2}, \frak h _{2}]$. Moreover, $[\frak h _{j}, \frak h _{j}]$ is an ideal, so it is invariant by $\Pi(\mathbf R ^r)$. Therefore $\Pi(U)Y$ and  $\Pi(V)X$ belong to $[\frak h_{j}, \frak h_{j}] $ for $j=1,2$. Since the intersection of the latter subspaces is $\{0\}$, the vectors $\Pi (U)Y$ and $\Pi (V)X$ are $0$. In other words, we have shown that
$$[\frak h_{1}, \frak h_{1}] \subset \bigcap _{V \in E _{2}} \Ker \Pi (V) ~\mathrm{~~and~~}~[\frak h_{2}, \frak h_{2}] \subset \bigcap _{U \in E _{1}} \Ker \Pi (U).$$
An easy computation shows that $\bigcap _{V \in E _{2}} \Ker \Pi (V) =\mathrm{Span}\{e _i : i \in I _1\}$, and similarly $\bigcap _{U \in E _{1}} \Ker \Pi (U) =\mathrm{Span}\{e _i : i \in I _2\}$. Since the subspaces $[\frak h _{1}, \frak h _{1}]$ and $[\frak h_{2}, \frak h_{2}]$ generate  $\mathbf R^n$, we finally obtain that $I _{1} \cup I _{2} =I$.

\noindent{\it Proof of (3)} : According to \cite{AW}, a connected Lie group $S$ admits a left-invariant non-positively curved Riemannian metric if, and only if, its Lie algebra $\frak s$ is an {\it NC algebra}. This means that $\frak s$ enjoys the following properties:
\begin{itemize}
\item [(i)] $\frak n :=[\frak s, \frak s]$ is a nilpotent ideal that is complemented in $\frak s$ by an abelian subalgebra $\frak a$.
\item [(ii)] There exists an element $\xi \in \frak a$, such that all the eigenvalues of $\mathrm{ad}\xi \restr _{\frak n}$ have negative real parts.
\item [(iii)] The action of $\frak a$ on $\frak n$ satisfies 3 additional conditions, which are automatically fulfilled when $\frak n$ is abelian and the $\frak a$-action is semisimple with real eigenvalues.
\end{itemize}
We refer to \cite[Definition 6.2]{AW} for the precise definition of NC algebra, and to the paragraph right after it for a discussion of the special cases.

Clearly the Lie algebra $\frak s _\alpha$ satisfies Items (i) and (iii). It satisfies (ii) if and only if there exists a $U \in \mathbf R ^r$ so that $\varpi _i (U) \leqslant -1$ for all $i \in I$. Let $v _i \in \mathbf R^r$ be such that $\varpi _i = \langle \cdot, v _i \rangle$, where $\langle \cdot , \cdot \rangle$ is a scalar product. One has $\varpi _i (U) \leqslant -1$ for every $i \in I$ if, and only if, every $v _i$ belongs to the subset defined by the inequality $\langle U, \cdot \rangle \leqslant -1$; {\it i.e.\!} to the affine half-space of $\mathbf R^r$, disjoint from $0$, and delimited by the hyperplane orthogonal to $U$ passing through $\frac{-U}{\Vert U \Vert ^2}$. The proof of Item (3) is now complete.

\noindent{\it Proof of (4)} : Let $\xi \in \mathbf R ^r$ be a non-zero vector. The associated subgroup $(\mathbf R \xi) \ltimes \mathbf R ^n$ is isomorphic to $\mathbf R \ltimes _\delta \mathbf R ^n$, where $\delta \in \mathrm{Der}(\mathbf R^n)$ is the derivation $\mathrm{diag}(\varpi _1(\xi), \dots , \varpi _n (\xi))$. Such a group is quasi-isometric to $\mathbb H ^{n+1} _{\mathbf R}$ if and only if $\delta$ is a multiple of $-I _n$ \cite{P2, CT}. Therefore the existence of a $\xi$ as in the statement, is equivalent to the existence of a $U \in \mathbf R ^r$ such that $\varpi _i (U) =-1$ for every $i \in I$. Consider again the vectors $v _i \in \mathbf R^r$ so that $\varpi _i = \langle \cdot, v _i \rangle$, where $\langle \cdot , \cdot \rangle$ is a scalar product. One has $\varpi _i (U) =-1$ for every $i \in I$ if, and only if, every $v _i$ belongs to the subset defined by the equation $\langle U, \cdot \rangle =-1$, {\it i.e.\!} to the affine hyperplane which is orthogonal to $U$ and which passes through $\frac{-U}{\Vert U \Vert ^2}$.
\end{proof}

\def\cprime{$'$}

\noindent Laboratoire Paul Painlev\'e, UMR 8524 CNRS / Universit\'e de Lille, 
Cit\'e Scientifique, B\^at. M2, 59655 Villeneuve d'Ascq, France. \\
E-mail: {\tt marc.bourdon@univ-lille.fr}.

\noindent Unit\'e de Math\'ematiques Pures et Appliqu\'ees, UMR 5669 CNRS / \'Ecole normale sup\'erieure de Lyon, 
46 all\'ee d'Italie, 69364 Lyon cedex 07,  France\\
E-mail: {\tt bertrand.remy@ens-lyon.fr}.

\end{document}